\documentclass[10pt]{amsart}

\usepackage{amsmath,amssymb,amscd,amsthm,latexsym,graphics,enumerate,engord}
\usepackage[all]{xy}
\usepackage{pstricks-add}

\theoremstyle{plain}
    \newtheorem{thm}{Theorem}[section]
    \newtheorem{lem}[thm]   {Lemma}
    \newtheorem{cor}[thm]   {Corollary}
    \newtheorem{prop}[thm]  {Proposition}

\theoremstyle{definition}
    \newtheorem{defn}[thm]  {Definition}
    
    \newtheorem{probs}[thm] {Problems}

    \newtheorem{ex}[thm]{Example}
    \newtheorem{exs}[thm]{Examples}
    \newtheorem{rem}[thm]{Remark}
    \newtheorem{rems}[thm]{Remarks}

\makeatletter
\newcommand{\pushright}[1]{\ifmeasuring@#1\else\omit\hfill$\displaystyle#1$\fi\ignorespaces}
\newcommand{\pushleft}[1]{\ifmeasuring@#1\else\omit$\displaystyle#1$\hfill\fi\ignorespaces}
\makeatother

\def\disfrac#1#2{{\displaystyle{\frac{#1}{#2}  }}}

\def\R{{\mathbb{R}}}

\def\Z{{\mathbb{Z}}}
\def\A{{\mathcal{A}}}

\def\cat{{\rm{cat}\hskip1pt}}
\def\secat{{\rm{secat}\hskip1pt}}

\def\TC{{\rm{TC}\hskip1pt}}
\def\hdim{{\rm{hdim}\hskip1pt}}

\def\red{}
\def\blue{}

\title[Hopf invariants for sectional category]{Hopf invariants for sectional category with applications to topological robotics}

\author{Jes\'us Gonz\'alez\textsuperscript{\dag}}
\author{Mark Grant}
\author{Lucile Vandembroucq\textsuperscript{\ddag}}
\thanks{\textsuperscript{\dag}~~Partially supported by Conacyt Research Grant 221221.}
\thanks{\textsuperscript{\ddag}~~Partially supported by the Research Centre of Mathematics of the University of Minho with the Portuguese Funds from the ``Funda\c c\~ao para a Ci\^encia e a Tecnologia'', through the Project PEstOE/MAT/UI0013/2014.}

\address{Departamento de Matem\'aticas, Centro de Investigaci\'on y de Estudios Avanzados del IPN, Av.~IPN 2508, Zacatenco, M\'exico City 07000, M\'exico}
\email{jesus@math.cinvestav.mx}

\address{Department of Mathematical Sciences, University of Aberdeen, Fraser Noble Building, Meston Walk, Aberdeen AB24 3UE, UK}
\email{mark.grant@abdn.ac.uk}

\address{Universidade do Minho, Centro de Matem\'atica, Campus de Gualtar, 4710-057 Braga, Portugal}
\email{lucile@math.uminho.pt}
\begin{document}

\begin{abstract}
We develop a theory of generalized Hopf invariants in the setting of sectional category. In particular we show how Hopf invariants for a product of fibrations can be identified as shuffle joins of Hopf invariants for the factors. Our results are applied to the study of Farber's topological complexity for 2-cell complexes, as well as to the construction of a counterexample to the analogue for topological complexity of Ganea's conjecture on Lusternik-Schnirelmann category.
\end{abstract}

\maketitle



\section{Introduction}

The sectional category of a fibration is the least number of open sets needed to cover the base, on each of which the fibration admits a continuous local section. This concept, originally studied by A.\ S.\ Schwarz \cite{Schwarz} under the name \emph{genus}, has found applications in diverse areas. Notable special cases include the Lusternik--Schnirelmann category (for which the standard reference has become the monograph \cite{CLOT} by Cornea, Lupton, Oprea and Tanr\'e) and Farber's topological complexity \cite{Far}, both of which are homotopy invariants of spaces which arise as the sectional category of associated path fibrations. The LS-category is classical and related to critical point theory, while topological complexity was conceived in the early part of the twenty-first century as part of a topological approach to the motion planning problem in Robotics. Further details and definitions will be given in Section \ref{secat} below. We remark that in the modern literature it is common to normalise these invariants so that the sectional category of a fibration with section is zero, a convention which we will adopt in this paper.

Most of the existing estimates for sectional category are cohomological in nature and are based on obstruction theory. The objective of the current paper is to produce more refined estimates using methods from unstable homotopy theory. There is an extensive literature on the application of generalized Hopf invariants to LS-category, originating with Berstein and Hilton \cite{B-H} and including spectacular applications by Iwase \cite{Iwase}, Stanley \cite{Stanley}, Strom \cite{Strom} and others (a nice summary can be found in Chapter 6 of \cite{CLOT}). Building on and generalizing the work of these authors, we develop a theory of generalized Hopf invariants in the setting of sectional category. \red{In very crude terms, Hopf invariants are homotopy classes which arise as obstructions for the sectional category of a map to increase by one upon adjunction of a cone on the base space. Although the explicit definition is based on general principles, the actual evaluation of such obstructions is a major problem in homotopy theory. A novel point in this paper is the development of an algebro-combinatorial method that allows us to get a reasonably sharp homological control of some of these obstructions. In fact, in many cases, Hopf invariants are carried by bottom cells, thus their evaluation can effectively be done in homological terms. The general idea is explained in Subsection~\ref{HopfProducts}, and full details supporting the theory are developed in Section~\ref{secsiete}. In particular, the method allows us to perform many} new computations of topological complexity which we believe would not be possible using \red{classic} obstruction-theoretic arguments.

Our first application is to the computation of the topological complexity of two-cell complexes $X=S^p\cup_\alpha e^{q+1}$. The LS-category of such a space $X$ is determined by the Berstein--Hilton--Hopf invariant
\[
H(\alpha)\in \pi_q(\Sigma\Omega S^p \wedge \Omega S^p) \cong \pi_q (S^{2p-1}\vee S^{3p-2}\vee\cdots )
\]
of the attaching map $\alpha: S^q\to S^p$ \cite{B-H}. When $p\ge2$, we have
\[
\cat(X) = \left\{\begin{array}{ll} 1 & \mbox{if }H(\alpha)=0, \\ 2 & \mbox{if }H(\alpha)\neq 0.
\end{array}\right.
 \]
 Note that in the metastable range $2p-1\le q\le 3p-3$ we may identify $H(\alpha)$ with its projection onto the bottom cell $H_0(\alpha) \in \pi_q(S^{2p-1})$. If $H_0(\alpha)\neq 0$ then $\cat(X) =2$, which by standard inequalities implies that $2\le \TC(X)\le 4$. In almost all cases, the usual cohomological bounds fail to determine the exact value of $\TC(X)$, for reasons of dimension. Using Hopf invariants, however, we are able to identify many cases with $\TC(X)\le 3$ (in Theorem~\ref{le3bis} below, where we use the symbol $\circledast$ to denote the join functor) as well as many cases with $\TC(X)\ge 3$ (in Theorem~\ref{ge3bis} below).

 \begin{thm}[Theorem \ref{le3}]\label{le3bis}
Let $X=S^p\cup_\alpha e^{q+1}$, where $\alpha: S^q\to S^p$ is in the metastable range $2p-1< q\le 3p-3$ and $H_0(\alpha)\neq 0$.
Then $\TC(X)\le 3$ if and only if $\hspace{.2mm}(4+2(-1)^p)H_0(\alpha)\circledast H_0(\alpha)=0$.
\end{thm}

\begin{thm}[Theorem \ref{ge3}]\label{ge3bis}
Let $X=S^p\cup_\alpha e^{q+1}$, where $\alpha: S^q\to S^p$ is in the metastable range $2p-1< q\le 3p-3$ and $H_0(\alpha)\neq 0$.
Then $\TC(X)\ge 3$ provided $(2+(-1)^p)H_0(\alpha)\neq0$.
\end{thm}

Note that the condition $(2+(-1)^p)H_0(\alpha)\neq0$ in Theorem~\ref{ge3bis} holds automatically if $p$ is odd. On the other hand, we remark right after the proof of Theorem~\ref{le3} that the condition $(4+2(-1)^p)H_0(\alpha)\circledast H_0(\alpha)=0$ in Theorem~\ref{le3bis} holds if $q$ is even.

Combining these two theorems gives the precise value $\TC(X)=3$ for large classes of two-cell complexes (see for instance Corollaries~\ref{partiuno} and~\ref{partidos}).
We are also able to draw conclusions about $\TC(X)$ outside of the metastable range, under the additional assumption $H(\alpha) = H_0(\alpha)$.

\begin{ex}\label{unocontc4}
If $p$ is odd, $2p-1<q\leq3p-3$, and the join square $H_0(\alpha)\circledast H_0(\alpha)$ is a non-trivial element of odd torsion, then $\TC(X)=4$.
\end{ex}

\begin{rem}
We also get a full description of $\TC(X)$ for $X=S^p\cup_\alpha e^{2p}$ (Theorems~\ref{grados} and~\ref{ClassicHopf} below). The proofs are, however, much more elementary than those in the cases of Theorems~\ref{le3bis} and~\ref{ge3bis}.
\end{rem}

Our second application is to the analogue of Ganea's conjecture for topological complexity. Recall that the product inequality $\cat(X\times Y)\le \cat(X) + \cat(Y)$ is satisfied by LS-category. Examples of strict inequality were given by Fox \cite{Fox}, involving Moore spaces with torsion at different primes. Ganea asked, in his famous list of problems \cite{Ganea1}, if we always get equality when one of the spaces involved is a sphere. That is, if $X$ is a finite complex, is it true that
 \[
 \cat(X\times S^k) = \cat(X) + 1 \qquad \mbox{for all $k\ge1$?}
 \]
 A positive answer became known as \emph{Ganea's conjecture}. The conjecture remained open for nearly $30$ years, shaping research in the subject. It was shown to hold for simply-connected rational spaces by work of Jessup~\cite{Jessup} and Hess~\cite{Hess}, and for large classes of manifolds by Rudyak~\cite{Rudyak}, Singhof~\cite{Singhof}, and Strom~\cite{strom-ls-on-manifolds}, until eventually proven to be false in general by Iwase \cite{Iwase,IwaseAinfty}. Iwase's counter-examples are two-cell complexes $X$ outside of the metastable range, whose Berstein--Hilton--Hopf invariants are essential but stably inessential, from which it follows that $\cat(X)=\cat(X\times S^k) =2$.

 The analogous question for topological complexity (which also satisfies the product inequality)
 asks whether, for any finite complex $X$ and $k\ge 1$, we always have an equality
\begin{equation}\label{TCGanea}
\TC(X\times S^k) = \TC(X) + \TC(S^k) = \left\{\begin{array}{ll} \TC(X) + 1 & \mbox{if $k$ odd,}\\ \TC(X) + 2 &\mbox{if $k$ even.} \end{array}\right.
\end{equation}
This question was raised by Jessup, Murillo and Parent \cite{JMP}, who proved that equation (\ref{TCGanea}) holds when $k\ge2$ for any formal, simply-connected rational complex $X$ of finite type. In this paper, we give a counter-example to (\ref{TCGanea}) for all even $k$, using Hopf invariant techniques.

\begin{thm}[Theorem \ref{noGaneaTC}]\label{noGaneaTCbis}
Let $Y$ be the stunted real projective space $\R P^6/\R P^2$, and let $X=Y\vee Y$. Then for all $k\ge2$ even,
\[
\TC(X)=4\quad\mbox{and}\quad \TC(X\times S^{k}) = 5.
\]
\end{thm}

We now briefly outline the contents of each section, and in doing so indicate the method of proof of the above results. Section~\ref{seccionpreliminar} is preliminary, and establishes our conventions and notations regarding base points, cones, suspensions and joins. In Section~\ref{secat} we give the necessary background on sectional category and relative sectional category, working in the context of fibred joins. The main new result here is Proposition \ref{secatupbyone}, which shows that the sectional category of a fibration relative to a subspace increases by at most one on attaching a cone, and moreover the section over the cone can be controlled in a certain sense. Section~\ref{hifscsec} is split into two subsections. In Subsection~\ref{HopfDefns} we define the Hopf invariants for sectional category, which determine the behaviour of relative sectional category under cone attachments. We also recall the definition of the Berstein--Hilton--Hopf invariants mentioned above and show how they fit into our framework. In Subsection~\ref{HopfProducts} we investigate Hopf invariants for cartesian products of fibrations. Using naturality of the exterior join construction, we prove our Theorem \ref{Hopfproducts} which states that Hopf invariants of a product can be obtained as joins of Hopf invariants of the factors, composed with a topological shuffle map (see below). This result is germane to the proofs of Theorems~\ref{le3bis} and~\ref{ge3bis} and Theorem~\ref{noGaneaTCbis}, which we give in Sections~\ref{secaptc2cw} and~\ref{secapgc4tc} respectively. Finally, Section~\ref{secsiete} is the technical heart of the paper. Using the description of the join in terms of simplicial (barycentric) parameters, together with the standard decomposition of the product of simplices $\Delta^n\times \Delta^m$ into simplices $\Delta^{n+m}$, we construct \emph{topological shuffle maps}
\[
\Phi^{A,B}_{n,m}: J^n(A)\circledast J^m(B)\to J^{n+m+1}(A\times B),
\]
which map from the join of the $(n+1)$-fold join of a space $A$ with the $(m+1)$-fold join of a space $B$ to the $(n+m+2)$-fold join of their product $A\times B$. We then describe (in Proposition~\ref{shuffleinhomology}) the effect of this map in homology, in terms of algebraic shuffles. This result is used in all of our calculations of Hopf invariants of products.

The idea of applying Hopf invariant techniques to obtain estimates for sectional category has been around for some time. For example, the upper bounds for the (higher) topological complexity of configuration spaces given in \cite{FG2} and \cite{GonzalezGrant}, and proved using obstruction theory, were originally obtained using Hopf invariants. \red{However, the homological method in this paper to evaluate certain Hopf invariants is new and independent of previous approaches.}

In developing the ideas in this paper, we have benefitted from discussions with many people. In this regard, the {authors would like to thank Michael Farber, Pierre Ghienne, Hugo Rodr\'iguez-Ord\'o\~nez and Enrique Torres-Giese. The authors are also} grateful to the organizers of the ESF ACAT conference and workshop ``Applied Algebraic Topology'' held in Castro Urdiales, Spain, from June~26 to July~5, 2014, where the final form of this project was shaped.

We conclude this introductory section with a list of open problems, some of which may be accessible by extending the techniques in this paper.

\begin{probs}
\begin{enumerate}[(a)]
\item Do there exist two-cell complexes $X=S^p\cup_\alpha e^{q+1}$ with $q>2p-1$ for which $\cat(X)=\TC(X)=2$? (The smallest such open case is that of $\alpha\in\pi_4(S^2)=\mathbb{Z}/2$ the generator, see Example~\ref{contieneposiblegap}(a).) Or for which  $\TC(X)=4$? (c.f.~Example~\ref{unocontc4}.)
\item Does every two-cell complex satisfy equation  (\ref{TCGanea})? What if $X$ is restricted to lie in the metastable range? (Ganea's conjecture does hold true for the latter complexes.)
\item Do there exist finite complexes $X$ and \emph{odd} $k$ for which $\TC(X\times S^k) = \TC(X)$?
\end{enumerate}
\end{probs}

\section{Preliminaries}\label{seccionpreliminar}
We work in the category of well-pointed\footnote{Base points will generically be denoted by an asterisk $\ast$.} compactly generated spaces having the homotopy type of CW-complexes. Thus, all maps, diagrams, and homotopies will be pointed, unless explicitly noticed otherwise. For instance, a homotopy section of a map $p\colon \mathcal{A}\to X$ is a pointed map $s\colon X \to\mathcal{A}$ with a pointed homotopy between $p\circ s$ and the identity on $X$. Products and mapping spaces are topologized in such a way that the product of two proclusions is a proclusion and evaluation maps are continuous. Fibrations are assumed to be pointed fibrations in the sense that they lift pointed homotopies or, equivalently, that they admit a pointed lifting function. Likewise, cofibrations are assumed to be pointed cofibrations\footnote{These are not actual restrictions in view of~\cite[Theorem~5.95]{MR2839990}.}.

We let $I$ stand for the unit interval $[0,1]$ with base point 0. For a pointed space $(A,\ast)$, we denote by $CA$ the cone $A\times I/A\times 1$ {with base point so that the projection $A\times I\to CA$ and the inclusion} $A\hookrightarrow CA$, $a\mapsto [a,0]$ {are pointed} (in general, the class of $(a,u)$ will be denoted by $[a,u]$). Note that the inclusion $A\hookrightarrow CA$ is a cofibration. The suspension of $A$ is defined by $\Sigma A:=CA/A$. We will also use the reduced suspension of $A$ given by the quotient $\widetilde{\Sigma} A:=\Sigma A/\ast\times I=A\wedge (I/\partial I)$. In both cases we take the obvious base points. The class of $[a,u]$ in both $\Sigma A$ and $\widetilde{\Sigma} A$ will also be denoted by $[a,u]$.

We denote by $\Delta^n$ the standard $n$-simplex of $\R^{n+1}$, given by
$$\Delta^n=\{(t_0,\ldots,t_{n})\in [0,1]^{n+1}\mid t_0+\cdots + t_{n}=1\}$$
with base point $(1,0,\ldots,0)$. The iterated $n$-fold reduced suspension of $A$ is homotopically equivalent to the quotient
$$ A\times \Delta^n/(A\times \partial \Delta^n \cup \ast \times \Delta^n)= A \wedge (\Delta^n/\partial \Delta^n)$$
that we will denote by $\widetilde{\Sigma}^nA$.

If $(B,\ast)$ is another pointed space we denote by $A\circledast B$ the join $CA\times B\;\cup\, A\times CB$ with base point $(\ast,\ast)\in A\times B\subset CA\times B\;\cap\, A \times CB$.
\begin{rems} \label{rmkpinchmap} We collect here some well-known facts which will be used in this work:
\begin{itemize}
\item[(a)] There is a canonical map given by the composition
$$\begin{array}{rcccl}
\zeta:A\circledast B &\to& \Sigma (A\times B) &\to & \widetilde{\Sigma} (A\times B)\\
\end{array}
$$
in which the second arrow is the identification map and the first arrow is the (non-pointed) map induced by the (non-pointed) maps
$$\begin{array}{rclcrcl}
A \times CB & \to& \Sigma (A\times B) & CA \times B & \to & \Sigma (A\times B)\\
(a,[b,u])&\mapsto &[(a,b),\frac{1-u}{2}] & ([a,u],b)&\mapsto &[(a,b),\frac{1+u}{2}]
\end{array}$$
Although the first map in the composition defining $\zeta$ is non-pointed, the composite $\zeta$ is pointed and, for this reason, we will mainly consider reduced suspensions in the sequel.
\item[(b)] The composition
$$
A\circledast B \stackrel{\zeta}{\longrightarrow} \widetilde{\Sigma} (A\times B) \longrightarrow \widetilde{\Sigma} (A\wedge B),
$$
where the second arrow is the identification map $[(a,b),u] \mapsto [a\wedge b,u]$, is a homotopy equivalence.
\item[(c)] The canonical (pointed) identification maps
$$\begin{array}{rclcrcl}
A \times CB & \to& \widetilde{\Sigma} (A\times B) & CA \times B & \to &\widetilde{\Sigma} (A\times B)\\
(a,[b,u])&\mapsto &[(a,b),u] & ([a,u],b)&\mapsto &[(a,b),u]
\end{array}$$
induce the following map
$$\begin{array}{rcl}
\nu:A\circledast B &\to& \widetilde{\Sigma} (A\times B) \vee \widetilde{\Sigma} (A\times B)\\
(a,[b,u]) &\mapsto & ([(a,b),u],*)\\
([a,u],b) &\mapsto & (*,[(a,b),u])\\
\end{array}
$$
which we call the {\em difference pinch map}. Indeed $\nu$ fits in the following commutative diagram
\begin{equation}\label{difzetadif}\xymatrix{
A\circledast B \ar[d]^{\zeta} \ar[rrd]^-{\nu}\\
\widetilde{\Sigma} (A\times B)\ar[d] \ar[rr]^-{\tilde{\nu}} &&\widetilde{\Sigma} (A\times B) \vee \widetilde{\Sigma} (A\times B)\ar[d]\\
\widetilde{\Sigma} (A\wedge B) \ar[rr]^-{\tilde{\nu}} &&\widetilde{\Sigma} (A\wedge B) \vee \widetilde{\Sigma} (A\wedge B)\\
}\end{equation}
where $\tilde \nu$ is the standard difference pinch map, that is the map
$$\begin{array}{rcl}
\tilde{\nu}:\widetilde{\Sigma} Z&\to& \widetilde{\Sigma} Z \vee \widetilde{\Sigma} Z\\
\left[z,u\right] &\mapsto & \left\{\begin{array}{lr}
([z,1-2u],*) & 0\leq u\leq 1/2,\\
(*,[z,2u-1]) & 1/2\leq u\leq 1.
\end{array}\right.
\end{array}
$$
In particular, if $f,g: \widetilde{\Sigma}Z \to X$ are two maps then the composite $\nabla \circ (f\vee g)\circ\tilde{\nu}$ is the difference $g-f$.
\end{itemize}
\end{rems}

\section{Sectional category and relative sectional category}\label{secat}

In this section we will review some known facts about sectional category, most of which can be found in the references \cite{CLOT}, \cite{James}, and \cite{Schwarz}. We will pay particular attention to the two most well-studied examples, namely Lusternik--Schnirelmann category and Farber's topological complexity.

In Definitions~\ref{definicionLScattegoria}--\ref{definicionsecat} below we do not require that subspaces are pointed. Consequently, the null-homotopies in Definition~\ref{definicionLScattegoria}, and the partial sections in Definitions~\ref{definicionTC} and~\ref{definicionsecat}, are not assumed to be pointed.

\begin{defn}\label{definicionLScattegoria} The \emph{(Lusternik--Schnirelmann) category} of a space $X$, denoted $\cat(X)$, is the {least integer} $n$ such that $X$ admits a cover by $n+1$ open sets $U_0, \ldots , U_n$ such that each inclusion $U_i\hookrightarrow X$ is null-homotopic. If no such integer exists, we set $\cat(X)=\infty$.
\end{defn}
\begin{defn} \label{definicionTC}The \emph{topological complexity} of a space $X$, denoted $\TC(X)$, is the least non-negative integer $n$ such that $X\times X$ admits a cover by $n+1$ open sets $U_0, \ldots , U_n$, on each of which there exists a continuous partial section of the evaluation fibration
\[
\pi_X:X^I\rightarrow X\times X, \qquad \alpha \mapsto \big(\alpha (0),\alpha (1)\big).
 \]
 Here $X^I$ denotes the space of paths in $X$ with base point the constant path in $\ast\in X$. If no such integer exists, we set $\TC(X)=\infty$.
\end{defn}

Both of these concepts are special cases of the sectional category of a fibration, first studied by A.\ S.\ Schwarz under the name \emph{genus}.
\begin{defn}\label{definicionsecat}
Let $p: \mathcal{A}\to X$ be a (surjective) fibration. The \emph{sectional category} of $p$, denoted $\secat(p)$, is defined to be the least non-negative integer $n$ such that $X$ admits a cover by $n+1$ open sets $U_0,\ldots , U_n$ on each of which there exists a continuous partial section of $p$. If no such integer exists, we set $\secat(p)=\infty$.
\end{defn}

One of the key results in this area is a characterization of sectional category in terms of fibred joins, which we now recall within our base-point setting. For $n\geq 0$ we denote by $p^{n+1}_X:\mathcal{A}^{n+1}_X\to X$ the fibred product of $n+1$ copies of $p$. The total space of the fibred join of $n+1$ copies of $p$ is the quotient space
$$J^n_X(\mathcal{A}) =\mathcal{A}_X^{n+1} \times \Delta^ n/\sim$$
where $\sim$ is the equivalence relation generated by
\[
(a_0,\ldots,a_i,\ldots, a_{n},t_0,\ldots,t_i,\ldots,
t_{n})\sim(a_0,\ldots,a'_i,\ldots,a_{n},t_0,\ldots,t_i,\ldots,t_{n})
\]
if $t_i=0$. We denote a general element of $\mathcal{A}_X^{n+1}\times \Delta^n$ by $(\mathbf{a},\mathbf{t})$ and its class in $J^n_X(\mathcal{A})$ by $\langle\mathbf{a}\mid \mathbf{t}\rangle$. In these terms, we choose
\begin{equation}\label{puntobasedeljoin}
\ast=\left\langle(\ast,\ldots,\ast) \;\left\rvert \,\left(1,0,\ldots,0\right)\right.\right\rangle
\end{equation}
as the base point in $J^n_X(\mathcal{A})$---naturally induced by the base points of $\mathcal{A}$ and $\Delta^n$. The map $p_n:J^n_X(\mathcal{A})\rightarrow X$ given by  $\langle\mathbf{a}\mid \mathbf{t}\rangle\mapsto p_X^{n+1}(\mathbf{a})$ is a (pointed) fibration, called the {\em $(n+1)$-fold fibred join of $p$}. If $A=p^ {-1}(\ast)$ is the fibre of $p$ over $\ast\in X$, then $p_n^{-1}(\ast)$ is the quotient $J^ n(A)=A^{n+1} \times \Delta^ n/\sim$ where the relation is the same as above. The latter space is homotopically equivalent to the $n$-fold suspension of the $(n+1)$-fold smash product of $A$ with itself. More precisely, with the notation introduced before, the following composite of identification maps
\begin{equation}\label{lasidentificacionesobvias}
J^n(A) \stackrel{r}{\longrightarrow}\widetilde{\Sigma}^nA^{n+1} \longrightarrow  \widetilde{\Sigma}^nA^{\wedge n+1}
\end{equation}
is a homotopy equivalence.

\begin{thm}\label{puntodepartida}
Let $p:\mathcal{A}\to X$ be a (surjective) fibration with $X$ paracompact. If $n\ge1$ or $\mathcal{A}$ is path-connected, then $\secat(p)\leq n$ if and only
if $p_n:J^n_X(\mathcal{A})\rightarrow X$ admits a (pointed) homotopy section.
\end{thm}

Most of the standard formulations of Theorem~\ref{puntodepartida} in the literature (e.g.~\cite[Theorem 3]{Schwarz}) are base-point free. In our context, the hypothesis ``$n\ge1$ or $\mathcal{A}$ path-connected'' in Theorem~\ref{puntodepartida} assures that the fiber of $p_n$ is path-connected. The pointed homotopy section (and even a pointed section) is then warranted since spaces are well pointed and $p_n$ is a pointed fibration.

\begin{rem}\label{fibredjoins}
For any $n\ge 0$ there is a commutative diagram
\begin{equation}\label{ijkappa}\xymatrix{
J^n(A) \ar[r]^{i_n}\ar@{^{(}->}[d] & J^n_X(\mathcal{A}) \ar@{^{(}->}[d]^{\jmath_n} \\
CJ^n(A) \ar[r]^{\kappa_n} &J^{n+1}_X(\mathcal{A}).
}\end{equation}
The inclusions $\jmath_n: J^n_X(\mathcal{A})\to J^{n+1}_X(\mathcal{A})$ are given by
$$\langle\mathbf{a}\mid \mathbf{t}\rangle\mapsto \langle\mathbf{a},a_{n+1}\mid \mathbf{t},0\rangle$$
where $\mathbf{a}\in \mathcal{A}_X^{n+1}$, $\mathbf{t}\in \Delta^n$ and $a_{n+1}$ is any element of $\mathcal{A}$ 
with $(\mathbf{a},a_{n+1})\in\mathcal{A}_X^{n+1}$.
They are compatible with the maps $p_n$ and $p_{n+1}$. The maps $\kappa_n: CJ^n(A) \to J^{n+1}_X(\mathcal{A})$ are given by $[\langle \mathbf{a}\mid \mathbf{t}\rangle,u]\mapsto \langle \mathbf{a},*\mid (1-u)\mathbf{t},u\rangle$. Notice that this map factors through the inclusion $J^{n+1}(A)\to J^{n+1}_X(\mathcal{A})$.
\end{rem}

Coming back to topological complexity and category, we suppose that $X$ is path-connected and paracompact with base point $\ast \in X$. Then we have
$$\TC(X)=\secat(\pi_X:X^I\rightarrow X\times X) \mbox{  \ \ and  \ \ } \cat(X)=\secat(p_X:PX\to X)$$  where $PX\subset
X^I$ is the space of paths beginning at the base point $*\in X$ and $p_X(\gamma)=\gamma(1)$.

The fibred join of $n+1$ copies of $p_X$ will be denoted by $g_n(X):G_n(X)\to X$ and referred as the {\em $n$-th Ganea fibration of $X$}. Thus $\cat(X)\leq n$ if an only if $g_n(X):G_n(X)\to X$ admits a (pointed) homotopy section. It is well-known that $G_1(X)\simeq \widetilde{\Sigma}\Omega X$ {with $g_1(X)$ homotopic to the adjoint of the identity map on $\Omega X$,} and that when $X={\widetilde{\Sigma}} A$ is a suspension {and $n\ge1$}, $g_n({\widetilde{\Sigma}} A)$ admits a canonical section given by the composition
\[
s_0: {\widetilde{\Sigma}} A \to {\widetilde{\Sigma}}\Omega{\widetilde{\Sigma}} A \simeq G_1({\widetilde{\Sigma}} A)\hookrightarrow G_n({\widetilde{\Sigma}} A)
\]
{where $\widetilde{\Sigma}A\to\widetilde{\Sigma}\Omega\widetilde{\Sigma}A$ is the suspension of the adjoint of the identity map on $\widetilde{\Sigma} A$.}

Likewise, the fibred join of $n+1$ copies of $\pi_X: X^I\to X\times X$ will be denoted by $g^{\TC}_n(X):G^{\TC}_n(X) \to X\times X$ and referred as the \emph{$n$-th $\TC$-Ganea fibration of $X$.} Thus $\TC(X)\leq n$ if and only if $g^{\TC}_n(X): G^ {\TC}_n(X) \to X\times X$ admits
a (pointed) homotopy section.

 Both $g_n(X)$ and $g_n^{\TC}(X)$ have as fibre (over the {corresponding} base points $\ast$ {and $(\ast,\ast)$}) the join $J^n(\Omega X)$ of $n+1$ copies of the based loop space $\Omega X$. When considering the Ganea fibrations, we denote this space by $F_n(X)$. The inclusions will be denoted {by}
$$i_n(X):F_n(X)\to G_n(X) \quad\mbox{and}\quad i^{\TC}_n(X):F_n(X)\to G^{\TC}_n(X).$$
As mentioned before, the inclusions
$$G_n(X)\hookrightarrow G_{n+1}(X) \quad\mbox{and}\quad  G^{\TC}_n(X)\hookrightarrow G^{\TC}_{n+1}(X)$$
correspond to inclusions on the first $n+1$ factors. The constructions $G_n$, $G_n^\TC$ and $F_n$ are homotopy functors.

\begin{prop}\label{TCcatHopf}
The map $\chi:P(X\times X)=PX\times PX \to X^I$ given by $(\alpha, \beta)\mapsto \alpha^ {-1}\beta$ induces a map $\bar{\chi}: \Omega(X\times X)=\Omega X\times \Omega X \to \Omega X$ and commutative diagrams for any $n$:

\[
\xymatrix{
F_n (X\times X) \ar[d]_{i_n(X\times X)} \ar[rr]^{\bar\chi_n} & &F_n( X) \ar[d]^{i^{\TC}_n(X)}\\
G_n(X\times X)\ar[rd]_{g_n(X\times X)} \ar[rr]^{\chi_n} &&G^{\TC}_n(X) \ar[ld]^{g^{\TC}_n(X)}\\
&X\times X.
}
\]
\end{prop}

This fact will be important in later sections, as it will allow us to construct Hopf invariants for $\TC(X)$ from Hopf invariants for $\cat(X\times X)$.

\medskip
Next we define relative sectional category and give some of its properties.

\begin{defn}\label{secatreldef}
Let $p:\mathcal{A}\to X$ be a fibration and let $\varphi: K\to X$ be any map. The {\em sectional category of $p: \mathcal{A}\to X$ relative to $\varphi\colon\thinspace K\to X$,} denoted by $\secat_\varphi(p)$, is the sectional category of $\varphi^*p$, the pullback of $p$ under $\varphi$. If $K\subseteq X$ and $\varphi$ is the inclusion, we denote $\secat_\varphi(p)=:\secat_K(p)$. In particular, $\secat_X(p)=\secat(p)$.
\end{defn}
\begin{prop}\label{secatprops}
Let $p:\mathcal{A}\to X$ be a (surjective) fibration with fibre $A$, and let $\varphi: K\to X$ be any map. The relative sectional category satisfies the following properties:
 \begin{enumerate}
 \item If $\psi: K\to X$ is homotopic to $\varphi$, then $\secat_\psi(p)=\secat_\varphi(p)$.
 \item\label{itemdosdel38} $\secat_\varphi(p)\le \secat(p)$.
 \item $\secat_\varphi(p)\le \cat(K)$.
 \item If $\pi_i(A)=0$ for $i<r$ then\footnote{We write $\hdim(K)$ for the \emph{homotopy dimension of $K$,} i.e.~the smallest dimension of CW complexes having the homotopy type of $K$.} $\secat_\varphi(p) \le \frac{\hdim(K)}{r+1}$.
 \item Suppose there are cohomology classes $x_1,\ldots , x_k\in H^*(X)$ with any coefficients such that $p^*(x_1) = \cdots = p^*(x_k)=0$ and $\varphi^*(x_1\cdots x_k)\neq 0$. Then $\secat_\varphi(p)\ge k$.
\end{enumerate}
In addition, if either $n\geq1$ or $\mathcal{A}$ is path-connected, then
\begin{enumerate}\addtocounter{enumi}{5}
 \item $\secat_\varphi(p)$ equals the smallest $n$ such that the map $\varphi: K\to X$ admits a (pointed) lift through $p_n: J^n_X(\mathcal{A})\to X$.
 \end{enumerate}
\end{prop}
If $\varphi: K\to X$ we denote $\secat_\varphi(p_X)$ by $\cat_\varphi(X)$, or by $\cat_K(X)$ when $\varphi$ is an inclusion. Similarly if $\varphi: K\to X\times X$ we denote $\secat_\varphi(\pi_X)$ by $\TC_\varphi(X)$, or by $\TC_K(X)$ when $\varphi$ is an inclusion. Note that {$\cat_\varphi(X)$ is usually denoted by $\cat(\varphi)$ in the literature ---the so-called category of the map $\varphi$---, {see for instance~\cite{CLOT}.} Likewise, the notation $\TC_K(X)$ differs from the notation for subspace topological complexity used, for example, in \cite{FarInv} and \cite{Grant}. We have proposed the slightly more specific notations $\cat_\varphi(X)$ and $\TC_\varphi(X)$ since much of the point in this paper is to benefit from comparing the sectional category of different fibrations over~$X$.}

The main advantage of relative sectional category over its absolute counterpart is monotonicity: if $K\subseteq K'\subseteq X$ and $p:\A\to X$ is a fibration, then $\secat_K(p)\le \secat_{K'}(p)$. Moreover, the relative sectional category either remains the same or increases by one on attaching a cell, or more generally on attaching a cone along a map. The following result and its proof are integral to the results in this paper.

\begin{prop}\label{secatupbyone}
Let $p: \mathcal{A} \to X$ be a fibration. Suppose that $X=K\cup_\alpha CS$ is the mapping cone of a map $\alpha: S\to K$. Then
\[
\secat_K(p)\le \secat(p) \le \secat_K(p)+1.
\]
\end{prop}
\begin{proof}
We may assume from the outset that $\alpha$ is a closed cofibration, and therefore an inclusion. This follows from Proposition \ref{secatprops} (1) and the standard construction, replacing $K$ with the mapping cylinder of $\alpha$. In particular, $S$, $CS$, $K$, and $X$ all share the base point $\ast=[\ast,0]$.

Assume $\secat_K(p)=n$, and choose a (pointed) lift $\phi: K\to J^{n}_X(\mathcal{A})$ of the inclusion $\iota: K\hookrightarrow X$ through the $(n+1)$-fold fibred join $p_n:J^{n}_X(\mathcal{A})\to X$. Let
\[
h_t : CS\to X = K\cup_S CS
\]
be a (pointed) homotopy which contracts the cone to its base point (indicated above). Then $\red{h'_t=h_t|_S} :S\to X$ is a (pointed) null-homotopy of $\iota\circ\alpha = p_n\circ \phi\circ \alpha$. Since $p_n$ is a (pointed) fibration, we can lift {$\red{h'_t}$} to a (pointed) homotopy \red{$\ell_t\colon S\to {J^n_X}(\mathcal{A})$} from $\red{\ell_0={}}\phi\circ \alpha$ to a map \red{$\ell_1\colon S\to {J^n_X}(\mathcal{A})$ which actually takes} values in $J^n(A)=p_n^{-1}(\ast)$. Since $\alpha$ is a {(pointed)} cofibration, \red{and since $\ell_0$ is extended by $\phi$,} we can extend \red{$\ell_t$} to a (pointed) homotopy $k_t:K\to J^n_X(\A)$ from $\red{k_0={}}\phi$ to a map $\red{k_1={}}\phi':K\to J^n_X(\A)$ \red{whose restriction} $\phi'\circ \alpha$ takes values in $J^n(A)$. We therefore have a strictly commuting square
\[
\xymatrix{
S \ar[r]^\alpha \ar[d]^H & K \ar[d]^{\phi'} \\
J^n(A) \ar[r]^{i_n} & J^n_X(\A)
}
\]
where the map $H$ is obtained by restriction of domain and codomain. Note that $p_n\circ k_t$ is a homotopy from $\iota$ to $p_n\circ \phi'$.

The maps
\[
CS \stackrel{CH\;}{\longrightarrow} CJ^n(A) \stackrel{\kappa_n}\longrightarrow J^{n+1}_X(\A)\qquad\mbox{and}\qquad K \stackrel{\phi'}{\longrightarrow} J^n_X(\A) \stackrel{\jmath_n}{\longrightarrow} J^{n+1}_X(\mathcal{A})
\]
agree on $S$ and so together define a map $\sigma: X\to J^{n+1}_X(\A)$. We claim that $\sigma$ is a homotopy section of $p_{n+1}:J^{n+1}_X(\A)\to X$, and therefore $\secat(p)\le n+1 = \secat_K(p)+1$.

To prove the claim we exhibit an explicit homotopy from the identity map of $X=K\cup_S CS$ to $p_{n+1}\circ \sigma$. {By construction,} the homotopies $h_t: CS\to X$ and $p_n\circ k_t: K\to X$ agree on $S$, and therefore glue together to give {a homotopy $H_t\colon X\to X$. It is clear that $H_0$ is the identity. On the other hand, $H_1=p_{n+1}\circ \sigma$ because, again by construction, both maps send $CS$ to the basepoint $\ast$ and are given by $p_{n+1}\circ\jmath_n\circ\phi' = p_n\circ \phi'$ on $K$.}

The inequality $\secat_K(p)\le \secat(p)$ comes directly from item~(\ref{itemdosdel38}) in Proposition~\ref{secatprops}.
\end{proof}

\begin{cor}\label{relsecatupbyone}
Let $p:\mathcal{A}\to X$ be a fibration, and let $\varphi:X'\to X$ be any map. Suppose that $X'=K\cup_\alpha CS$ is the mapping cone of a map $\alpha: S\to K$. Then
\[
\secat_{\varphi|_K}(p)\le \secat_\varphi(p) \le \secat_{\varphi|_K}(p)+1.
\]
\end{cor}
\begin{proof} This is Proposition \ref{secatupbyone} applied to $\varphi^* p$.
\end{proof}

\begin{cor}\label{amejorarse}
For any fibration $p: \A\to X$ we have $\secat(p)\le \operatorname{cl}(X)$ where $\operatorname{cl}(X)$ denotes the cone length of $X$.
\end{cor}

Corollary~\ref{amejorarse} is of course improved by the standard estimate $\secat(p)\le\cat(X)$. The real strength of Proposition~\ref{secatupbyone} will become apparent with the constructions in the next section.

\begin{rem}\label{neuralgico}
We could have given a simpler proof of Proposition~\ref{secatupbyone} using Lemma \ref{extendsection} in the next section. Note however that the above proof furnishes a strictly commuting cubical diagram
$$ 
\xymatrix{
 & S \ar[rr]^{\alpha} \ar[dl]_{H} \ar@{^{(}->}'[d][dd] & & K \ar[dd]^\iota \ar[dl]_{\phi'} \\
J^n(A)\ar[rr] \ar@{^{(}->}[dd] & & J^n_X({\mathcal A}) \ar[dd] & \\
 & CS \ar'[r][rr] \ar[ld]^{CH} & & X \ar[ld]_{\sigma} \\
CJ^n(A)\ar[rr]^{\kappa_n} & & J^{n+1}_X({\mathcal A}) &}
$$ 
in which $\phi'$ is a (pointed) homotopy lifting of $\iota$ through $p_n$ and $\sigma$ is a (pointed) homotopy section of $p_{n+1}$. This diagram will be especially important in what follows.
\end{rem}

\section{Hopf invariants for sectional category}\label{hifscsec}
In this section we introduce the Hopf invariants which determine whether the relative sectional category increases on attaching a cone. We give the main definitions in {Subsection} \ref{HopfDefns}, and then in {Subsection} \ref{HopfProducts} prove a fundamental result about Hopf invariants of product fibrations.

\subsection{Definitions}\label{HopfDefns}

Before giving the definition of the Hopf invariants considered in this paper, we record a couple of technical lemmas.

\begin{lem}\label{fibjoinssplit}
Given any fibration $p:\mathcal{A}\to X$ and $n\geq 1$, the $(n+1)$-fold fibred join $p_n:J^n_X(\mathcal{A})\rightarrow X$ splits after looping once. Consequently, if $Y=\widetilde{\Sigma}S$ is a suspension, then the induced maps of (pointed) homotopy groups $(p_n)_*:[Y,J^n_X(\mathcal{A})]\to [Y,X]$ and $(i_n)_*:[Y,J^n(A)]\to [Y,J^n_X(\mathcal{A})]$ are split surjective and split injective, respectively.
\end{lem}

\begin{proof}
Using a (pointed) lifting function for $p$ we may construct a map $\chi: PX\to \mathcal{A}$ rendering the following diagram commutative:
    \[
    \xymatrix{
    PX  \ar[rr]^{\chi} \ar[rd]_{p_X} & & \mathcal{A} \ar[ld]^p \\
      &  X &
      }
      \]
      The fibred join construction is functorial for fibrewise maps, and so we obtain diagrams
         \[
    \xymatrix{
    G_n(X)  \ar[rr] \ar[rd]_{g_n(X)} & & J^n_X(\mathcal{A}) \ar[ld]^{p_n} \\
      &  X &
      },
      \qquad
       \xymatrix{
    \Omega G_n(X)  \ar[rr] \ar[rd]_{\Omega g_n(X)} & & \Omega J^n_X(\mathcal{A}) \ar[ld]^{\Omega p_n} \\
      &  \Omega X &
      }
      \]
      for each $n\ge 0$. Since $\Omega g_n(X)$ admits a homotopy section for $n\ge1$ (see \cite[Exercise 2.1]{CLOT}, for instance) so does $\Omega p_n$, and the result follows.
\end{proof}
\begin{lem}\label{extendsection}
Let $S\stackrel{\alpha}{\to} K\stackrel{\iota}{\to} X=K\cup_\alpha CS$ be a cofibration sequence, and let $\rho: Z\to X$ {and $\phi: K\to Z$ be maps with homotopies} $\phi\circ \alpha\simeq *$ and $\rho\circ \phi \simeq \iota$. If $\rho_\ast:[{\widetilde{\Sigma}} S,Z]\to [{\widetilde{\Sigma}} S,X]$ is surjective, then there exists a map $\sigma: X\to Z$ and pointed homotopies $\sigma\circ \iota \simeq \phi$ and $\rho\circ \sigma \simeq \mathrm{Id}_X$.
\end{lem}
\begin{proof}
This is a slight generalization of \cite[Lemma 6.28]{CLOT}, with the same proof.
\end{proof}

\begin{defn}\label{Hopfinv}
Let $p:\mathcal{A}\to X$ be a fibration. Suppose that $X=K\cup_\alpha CS$ is the mapping cone of a map $\alpha: S\to K$. Suppose also that $\secat_K(p)\le n$, and let $\phi: K\to J^n_X(\mathcal{A})$ be a (pointed) homotopy lifting of the inclusion $\iota: K\hookrightarrow X$ through $p_n: J^n_X(\A)\to X$. As in the proof of Proposition \ref{secatupbyone}, {consider a pointed-}homotopy commutative diagram
\[
\xymatrix{
S \ar[r]^\alpha \ar[d]^H & K \ar[d]^{\phi} \\
J^n(A) \ar[r]^{i_n} & J^n_X(\A).
}
 \]
 If $n\ge 1$ and $S$ is a {reduced} suspension, then by Lemma \ref{fibjoinssplit} the pointed-homotopy class of the map $H: S \to J^n(A)$ depends only on the pointed-homotopy classes of $\alpha$ and~$\phi$. Any representative of this class will be denoted $\red{H(p)={}}H^n_{\phi,\alpha}(p)$ and called the {\em Hopf invariant} associated to the data $(p,n,\phi,\alpha)$. The set of all such Hopf invariants as $\phi$ ranges over all possible (pointed) lifts is denoted $\red{\mathcal{H}(p)={}}\mathcal{H}_{\alpha}^n(p)$ and called the {\em Hopf set} associated to $(p,n,\alpha)$.
 \end{defn}

 \begin{prop}\label{caracterizacionshida}
 Under the conditions of Definition \ref{Hopfinv}, we have $\secat(p)\le n$ if and only if the Hopf set $\mathcal{H}^n_\alpha(p)$ contains the trivial element.
 \end{prop}
 \begin{proof}
 Suppose there is a lift $\phi: K\to J^n_X(\mathcal{A})$ such that the associated Hopf invariant $H^n_{\phi,\alpha}(p):S\to J^n(A)$ is null-homotopic. Then $\phi\circ\alpha\simeq *$ and so Lemma \ref{extendsection} gives a homotopy section $\sigma$ of $p_n$ extending $\phi$ {up to pointed homotopy.}

 Conversely, suppose that $\secat(p)\le n$ and let $\sigma: X\to J^n_X(\mathcal{A})$ be a homotopy section of $p_n$. Then $\phi:=\sigma\circ\iota: K\to J^n_X(\mathcal{A})$ is a homotopy lifting of $\iota$ through $p_n$ satisfying $\phi\circ\alpha\simeq \ast$, whose associated Hopf invariant is therefore trivial (as explained in Definition~\ref{Hopfinv}, the last conclusion uses Lemma~\ref{fibjoinssplit}).
 \end{proof}

\begin{prop}\label{uniqueness}
Under the conditions of Definition \ref{Hopfinv}, suppose that the fibre $A$ is $(r-1)$-connected with $r\ge1$. If $\hdim(K)<(n+1)r+n$, then the Hopf set $\mathcal{H}^n_\alpha(p)$ consists of a single element.
\end{prop}
\begin{proof}
Note that we are assuming $\secat_K(p)\le n$, so that the Hopf set is non-empty. If $A$ is $(r-1)$-connected then $J^n(A)$ is $\big((n+1)r + n-1\big)$-connected. Then $p_n$ is an $\big((n+1)r + n\big)$-equivalence, and it follows that the induced map $(p_n)_*:[K,J^n_X(\mathcal{A})]\to [K,X]$ is bijective when
$\hdim(K)<(n+1)r + n$ (see~\cite[Corollary 7.6.23]{Spanier}). Thus the lifting $\phi$ of $\iota$ is unique up to homotopy.
\end{proof}

\begin{ex}[Berstein--Hilton--Hopf invariants \cite{B-H,CLOT,Iwase,IwaseAinfty,Stanley}]\label{ejemploBHH}
Let $K$ be a path-connected space with $\cat(K)\le n\ge 1$, and let $\alpha: S^q\to K$ be a map with $q\ge 1$. The cofiber $X=K\cup_\alpha CS$ of $\alpha$ satisfies $\cat(X)\le n+1$. Berstein and Hilton introduced in~\cite{B-H} generalized Hopf invariants to detect whether $\cat(X)\le n$. Here we give the modification of their definition used by Iwase in \cite{Iwase}.

Let $s: K\to G_n(K)$ be a (pointed) section of $g_n(K):G_n(K)\to K$, the $n$-th Ganea fibration of $K$. Then we define:
\begin{itemize}
\item $H_s'(\alpha):=s\circ \alpha -G_n(\alpha)\circ s_0 \in \pi_{q}(G_n(K))$ where $s_0$ is the canonical section of $g_n(S^q)$.
\item $H_s(\alpha)\in \pi_q(F_n(K))$ the unique (up to homotopy) map satisfying $H_s'(\alpha)= i_n(K)\circ H_s(\alpha)$.
\end{itemize}

Both of these elements will be called the \emph{Berstein--Hilton--Hopf invariant of $\alpha$ associated to $s$}. The set of such elements as $s$ ranges over all (homotopy classes of) such sections is denoted $\mathcal{H}'(\alpha)\subseteq \pi_q(G_n(K))$ or $\mathcal{H}(\alpha)\subseteq \pi_q(F_n(K))$, and called the \emph{Berstein--Hilton--Hopf set of $\alpha$}. If $K$ is a CW-complex, it is shown in~\cite[Section~6.4]{CLOT} that
\begin{equation}\label{recuperarclot}
\mbox{{\it $\cat(X)\le n$ if and only if $\,0\in \mathcal{H}(\alpha)$, provided $\max\{\dim(K),2\}\le q$.}}
\end{equation}

We now make explicit the relationship of these Berstein--Hilton--Hopf invariants with the Hopf invariants discussed in this section. Let $\iota: K\to X$ denote the inclusion into the cofiber. By Proposition \ref{secatprops} we have $\cat_K(X)\le\cat(K)\le n$, and indeed
\begin{equation}\label{primerafactorizacion}
\phi=G_n(\iota)\circ s:K\to G_n(X)
\end{equation}
is a (pointed) lifting of $\iota$ through $g_n(X)$. Then the associated Hopf invariant $H^n_{\phi,\alpha}(g_0(X)):S^q\to F_n(X)$ satisfies $H^n_{\phi,\alpha}(g_0(X)) = F_n(\iota)\circ H_s(\alpha)$. This follows from the definitions, together with the diagram
\[
\xymatrix{
F_n(K) \ar[r]^{i_n(K)} \ar[d]_{F_n(\iota)} & G_n(K) \ar[d]^{G_n(\iota)} \\
F_n(X) \ar[r]^{i_n(X)} & G_n(X)
}
\]
and the observation that $G_n(\iota)\circ H_s'(\alpha)\simeq G_n(\iota)\circ s\circ\alpha$ since $G_n$ is a homotopy functor and $\iota\circ\alpha$ is null-homotopic. {Thus}
\begin{equation}\label{lainclusion}
F_n(\iota)_\ast\big( \mathcal{H}(\alpha)\big)\subseteq \mathcal{H}^n_{\alpha}(g_0(X)).
\end{equation}

Finally, we note that {(\ref{lainclusion}) can be improved to an equality under the hypothesis in~(\ref{recuperarclot}). Namely, if} $K$ is a connected complex {and $q\ge\max\{\dim(K),2\}$, then} the maps $F_n(\iota):F_n(K)\to F_n(X)$ and $G_n(\iota):G_n(K)\to G_n(X)$ are $(q+1)$-equivalences for all $n\ge 1$, by \cite[Lemma 6.26]{CLOT}. Consequently, any (pointed) homotopy lift $\phi$ of $\iota$ through $g_n(X)$ arises as in~(\ref{primerafactorizacion}) for some $s$, so that in fact $F_n(\iota)_\ast\big( \mathcal{H}(\alpha)\big)= \mathcal{H}^n_{\alpha}(g_0(X))$. Furthermore, the triviality of $H^n_{\phi,\alpha}(g_0(X))$ is equivalent to the triviality of $H_s(\alpha)$, so that Proposition~\ref{caracterizacionshida} recovers~(\ref{recuperarclot}).
\end{ex}

\begin{ex}[cat-Hopf invariants of spheres]\label{cathopfsphere}
Let $q\ge 2$. We may regard the $q$-sphere $S^q$ as the cofiber $C^-S^{q-1} \cup_\alpha CS^{q-1}$ of the inclusion $\alpha: S^{q-1}\hookrightarrow C^-S^{q-1}$ of the base of the cone $C^-S^{q-1}=S^{q-1}\times [-1,0]/S^{q-1}\times\{-1\}$. The base point of this cone is $[\ast,0]$, where $\ast$ is the base point of $S^{q-1}$, and $\alpha$ is a pointed cofibration. Here, $\cat_{C^-S^{q-1}}(S^q)=0$ and $\cat(S^q)=1$. Since $n=0$ and the fibration $g_0(S^q): G_0(S^q)\to S^q$ does not split after looping, the uniqueness statement in Definition~\ref{Hopfinv} breaks down. We can, however, define a Hopf invariant $H^0(S^q)$ as follows.

 Fix a pointed homotopy equivalence $\xi: S^q\to \widetilde \Sigma S^{q-1}$ between $S^q$ defined as above and the reduced suspension of $S^{q-1}$, and let $\xi^{-1}$ denote a {pointed} homotopy inverse of $\xi$. Denote by $\eta: S^{q-1}\to \Omega\widetilde\Sigma S^{q-1}$ the standard adjunction, given by $\eta(x)(t) = [x,t]$. We define $\tilde\eta: S^{q-1}\to F_0(S^q)$ to be the composition
  $$S^{q-1} \stackrel{\eta}{\to} \Omega \widetilde \Sigma S^{q-1} \stackrel{\Omega \xi^{-1}}{\to} \Omega S^q = F_0(S^q),$$
 and note that $\tilde\eta$ is a pointed map.

    With these preliminaries, we can construct a commuting diagram
\[
\xymatrix{
 & S^{q-1} \ar[rr] \ar[ld]_{\tilde\eta} \ar@{^{(}->}'[d][dd] & & C^-S^{q-1} \ar[dd] \ar[ld]_{\phi'} \\
 F_0(S^q)\ar[rr] \ar@{^{(}->}[dd] & & G_0(S^q) \ar[dd] & \\
& CS^{q-1} \ar'[r][rr] \ar[ld]_{C\tilde\eta} & & S^q \ar[ld]_{\sigma}\\
CF_0(S^q) \ar[rr] & & G_1(S^q) &
}
 \]
 where $\phi'([x,u])(t) = \tilde\eta(x)\big( (1+u)t\big)$, {which is of course a homotopy lifting of the inclusion $C^-S^{q-1}\hookrightarrow S^q$ through $g_0(S^q)$,} and
 \[
 \sigma([x,u]) =\left\{
\begin{array}{lr}
\left\langle \hspace{.1mm}{\phi'([x,u])
\hspace{.3mm}, \hspace{.2mm}
\phi'([x,u])}
\mid 1,0\rule{0mm}{3.2mm}\right\rangle, & -1\leq  u\leq 0; \\
\left\langle \tilde\eta(x), *\mid 1-u,u\rule{0mm}{3.2mm}\right\rangle, & 0\leq u\leq 1.
\end{array}
\right.
\]
It is straightforward to check that the diagram commutes and $\sigma: S^q\to G_1(S^q)$ is a homotopy section of $g_1(S^q)$. {The situation is now analogous to that in Remark~\ref{neuralgico} and we} may therefore consider the homotopy class of $\tilde\eta$ as the Hopf invariant $H^0(S^q)$. {Note that, by construction, $H^0(S^q)$ is homotopic to the adjoint of the identity on $\widetilde{\Sigma}S^{q-1}$ and, therefore, can be identified up to homotopy {with} the inclusion of the bottom cell in $F_0(S^q)=\Omega S^q\simeq S^{q-1}\cup e^{2(q-1)}\cup\cdots$.}
\end{ex}

\subsection{Products} \label{HopfProducts}
Let $p:{\mathcal A}\to X$ and $q:{\mathcal B}\to Y$ be fibrations with respective fibres $A$ and $B$. \red{The goal of this section ---and a major accomplishment of this paper--- is the explicit construction (in Theorem~\ref{Hopfproducts} below) of a Hopf invariant $H(p\times q)$ of the product $p\times q: \mathcal{A}\times\mathcal{B}\to X\times Y$ in terms of given Hopf invariants $H(p)$ and $H(q)$ of the factors $p$ and $q$. The strength of the construction comes from the tight homology control we get on $H(p\times q)$ (Theorems~\ref{eli} and~\ref{morhipsmafterdesusp} below).}

\medskip
\red{In slightly more detail,} \blue{for each pair of non-negative integers $n$ and $m$, \red{there are} maps $\Phi^{A,B}_{n,m}$ and $\Psi^{\mathcal{A},\mathcal{B}}_{n,m}$ fitting in \red{a} commutative diagram
\begin{equation}\label{shufflewjoins}
\xymatrix{
J^ n(A)\circledast J^m(B)\ar[rr]^-{{\Phi_{n,m}^{A,B}}} \ar[d]_{{W}}&& J^{n+m+1}(A\times B)\ar[d]^{{i_{n+m+1}}} \\
J^{n+1}_X({\mathcal A})\times J_Y^{m}({\mathcal B})\cup J_X^{n}({\mathcal A})\times J_Y^{m+1}({\mathcal B})\ar[rr]^-{{\Psi_{n,m}^{{\mathcal A},{\mathcal B}}}} \ar[d]_{{W'}}&& J_{X\times Y}^{n+m+1}({\mathcal A}\times {\mathcal B})\ar[d]^{{(p\times q)_{n+m+1}}} \\
J^{n+1}_X({\mathcal A})\times J_Y^{m+1}({\mathcal B})\ar[rr]^{{p_{n+1}\times q_{m+1}}} &&X\times Y}
\end{equation}
where $W$ and $W'$ are \red{obvious maps induced by the inclusions in~(\ref{ijkappa})\,---\,c.f.~(\ref{doswhiskersobvios}) below.} For Ganea fibrations, such diagrams (\ref{shufflewjoins}) have already been (at least implicitly) constructed and used, in particular in \cite{Iwase} and \cite{IwaseAinfty}. \red{In} Section~\ref{secsiete}, we give a \red{new,} explicit, and natural construction of the maps $\Phi^{A,B}_{n,m}$ and $\Psi^{\mathcal{A},\mathcal{B}}_{n,m}$, which exhibits the \red{former one} as a sort of topological shuffle map, enabling us to understand the effect of $\red{\Phi_{n,m}^{A,B}}$ in homology. \red{In particular:}
\begin{thm}\label{eli}
Let $p\ge2$. The degree $d$ of the map $S^{(n+m+2)p-1}\to S^{(n+m+2)p-1}$ induced by restriction of \red{the composition}
$$ \xymatrix{
F_n(S^p)\circledast F_m(S^p) \ar[r]^{\;\;\,\Phi_{n,m}^{\Omega S^p,\Omega S^p}} & F_{n+m+1}(S^p \times S^p) \ar[r]^{\hspace{3.4mm}\bar{\chi}_{n+m+1}} & F_{n+m+1}(S^p)
}$$
to the bottom \red{cell} is given by
$$\pm d=\#(S^+_{n+1,m+1})+(-1)^p\#(S^-_{n+1,m+1})$$
where $S^+_{n+1,m+1}$ (resp.~$S^-_{n+1,m+1}$) stands for the set of $(n+1,m+1)$ shuffles of positive (resp.~negative) signature.
\end{thm}
\red{The map $\bar{\chi}_{n+m+1}$ has been introduced in Proposition~\ref{TCcatHopf}. The detailed construction of the maps in~(\ref{shufflewjoins}), and the proof of Theorem~\ref{eli} is postponed to Section~\ref{secsiete}. Here we highlight a few key instances.}
\begin{exs}\label{degree-2cells}
\begin{enumerate}[(a)]
\item
The map $S^{4p-1}\to S^{4p-1}$ induced by restriction to the bottom cell of the composite
$$\xymatrix{
F_1(S^p)\circledast F_1(S^p) \ar[rr]^-{\Phi_{1,1}^{\Omega S^p,\Omega S^p}} && F_3(S^p \times S^p) \ar[r]^-{\bar{\chi}_3} & F_3(S^p)
}$$
has degree $\pm(4+2(-1)^p)$.
\item The composite
$$\xymatrix{
F_1(S^p)\circledast F_0(S^p) \ar[rr]^-{\Phi_{1,0}^{\Omega S^p, \Omega S^p}} && F_2(S^p\times S^p) \ar[r]^-{\bar{\chi}_2} & F_2(S^p)
}$$
induces a map $S^{3p-1}\to S^{3p-1}$ by restriction to the bottom cell. The degree of this map is $\pm(2+(-1)^p)$.
\item The composite
$$\xymatrix{
F_0(S^p)\circledast F_0(S^p) \ar[rr]^-{\Phi_{0,0}^{\Omega S^p, \Omega S^p}} && F_1(S^p\times S^p) \ar[r]^-{\bar{\chi}_1} & F_1(S^p)
}$$
induces a map $S^{2p-1}\to S^{2p-1}$ by restriction to the bottom cell. The degree of this map is $\pm(1+(-1)^p)$.
\end{enumerate}
\end{exs}}

\red{We now explain how the maps $\Phi_{n,m}^{A,B}$ can be used to construct (in Theorem~\ref{Hopfproducts} below) a Hopf invariant $H(p\times q)$ out of Hopf invariants $H(p)$ and $H(q)$.} We begin by recalling the exterior join construction \cite{Bau,Marcum,Stanley}. {The following considerations are meant to prepare grounds for Proposition~\ref{exteriorjoin} below, which is} a slight generalization of \cite[Proposition 2.9]{Stanley}.

Suppose we are given commutative diagrams
\begin{equation}\label{10i}
\xymatrix{
S(i) \ar[r]^{f(i)} \ar@{^{(}->}[d] & K(i) \ar[d]^{g(i)}\\
CS(i) \ar[r]_{F(i)} & X(i)
}
\end{equation}
($i=1,2$), and consider the commutative cube
\[
\xymatrix{
 & S(1)\times S(2) \ar[ld]_{f(1)\times f(2)\hspace{3mm}} \ar@{^{(}->}'[d][dd] \ar@{^{(}->}[rr] & & CS(1)\times S(2) \ar[ld]^{\hspace{3mm}F(1)\times f(2)} \ar@{^{(}->}[dd] \\
K(1)\times K(2) \ar[rr]^{\hspace{2cm}g(1)\times1} \ar[dd]_{1\times g(2)} &
& X(1)\times K(2) \ar[dd]_>>>>>>{1\times g(2)} & \\
 & S(1)\times CS(2) \ar@{^{(}->}'[r][rr] \ar[ld]_{f(1)\times F(2)\hspace{2mm}} & & CS(1)\times CS(2) \ar[dl]^{\hspace{3mm}F(1)\times F(2)} \\
K(1) \times X(2) \ar[rr]_{g(1)\times1} & & X(1)\times X(2). &
}
\]
Recall $S(1)\circledast S(2)$ is the push-out of the ``back'' face of the above cube, with whisker map given by the obvious inclusion $S(1)\circledast S(2)\hookrightarrow CS(1)\times CS(2)$. Likewise, let $X(1)\times K(2)\cup K(1)\times X(2)$ stand for the push-out of the ``front'' face of the above cube, and let
\begin{equation}\label{doswhiskersobvios}
S(1)\circledast S(2) \stackrel{W}{\longrightarrow} X(1)\times K(2)\cup K(1)\times X(2) \stackrel{W'}{\longrightarrow} X(1)\times X(2)
\end{equation}
be the obvious whisker maps.
We obtain a commutative diagram
\begin{equation}\label{diagresultante}\xymatrix{
S(1)\circledast S(2) \ar[rr]^-{W}\ar@{^{(}->}[d] & &  X(1)\times K(2) \cup K(1) \times X(2) \ar[d]^{{W'}} \\
CS(1)\times CS(2) \ar[rr]^{{F(1)\times F(2)}} & & X(1)\times X(2).
}\end{equation}

\begin{prop}\label{exteriorjoin}
If the two initial diagrams~(\ref{10i}) are homotopy push-outs {(respectively, strict push-outs),} then so is~{(\ref{diagresultante}).} Furthermore, this construction is natural in the following sense. Suppose for $i=1,2$ there are maps $H(i): S(i)\to {\overline{S(i)}}$, $\phi(i): K(i)\to {\overline{K(i)}}$ and $\sigma(i): X(i)\to {\overline{X(i)}}$ {fitting in the commutative cube}
\[
\xymatrix{
 & S(i) \ar[ld]^{H(i)} \ar@{^{(}->}'[d][dd] \ar[rr] & & K(i) \ar[ld]^{\phi(i)} \ar[dd] \\
{\overline{S(i)}} \ar[rr] \ar@{^{(}->}[dd] & & {\overline{K(i)}} \ar[dd] & \\
 & CS(i) \ar'[r][rr] \ar[ld]^{CH(i)} & & X(i) \ar[dl]^{\sigma(i)} \\
C{\overline{S(i)}} \ar[rr] & & {\overline{X(i)}} &
}
\]
Then there results a commutative cube
\[
\resizebox{1\textwidth}{!}{\xymatrix{
 & S(1)\circledast S(2) \ar[dl]_{H(1)\circledast H(2)\hspace{3mm}} \ar@{^{(}->}'[d][dd] \ar[rr] & & X(1)\times K(2) \cup K(1)\times X(2) \ar[dl]^{\hspace{3mm}{W''}} \ar[dd] \\
{\overline{S(1)}}\circledast {\overline{S(2)}} \ar[rr] \ar@{^{(}->}[dd] & & {\overline{X(1)}} \times {\overline{K(2)}} \cup {\overline{K(1)}} \times {\overline{X(2)}} \ar[dd] & \\
 & CS(1)\times CS(2) \ar'[r][rr] \ar[dl]^>>>>>>>>>>>>>>{\hspace{6mm}CH(1)\times CH(2)} & & X(1)\times X(2) \ar[dl]^{\hspace{3mm}\sigma(1)\times\sigma(2)} \\
C{\overline{S(1)}}\times C{\overline{S(2)}} \ar[rr] & & {\overline{X(1)}} \times {\overline{X(2)}} &
}}
\]
where $W''$ is the whisker map comming from the assumption that $X(1)\times K(2) \cup K(1)\times X(2)$ is a pushout.
\end{prop}

\begin{rem}\label{conosenelextjoicon}
As shown in~\cite[Lemma 2.8]{Stanley}, $CS(1) \times CS(2)$ is homeomorphic to $C(S(1) \circledast S(2))$ in such a way that the inclusion $S(1) \circledast S(2) \hookrightarrow CS(1) \times CS(2)$ corresponds to the inclusion of the base of the cone.
\end{rem}

The following is a generalisation to arbitrary fibrations of a result due to Iwase (\cite[Proposition 5.8]{Iwase}, \cite[Theorem 5.5]{IwaseAinfty}, see also \cite{Harper}).

\begin{thm}\label{Hopfproducts}
Let $p:{\mathcal A}\to X$ and $q:{\mathcal B}\to Y$ be fibrations with respective fibres $A$ and $B$. Suppose that $X=K\cup_\alpha CS$ is the cofibre of a map $\alpha: S\to K$, and $Y=L\cup_\beta CT$ is the cofibre of a map $\beta: T\to L$, where $S$ and $T$ are reduced suspensions. If $\secat_K(p)\le n$ with Hopf invariant $H^n_{\phi,\alpha}(p):S\to J^n(A)$ {associated to a (pointed) homotopy lifting $\phi \colon K \to J^n_X(\mathcal{A})$ of the inclusion $K\hookrightarrow X$ through $p_n$,} and $\secat_L(q)\le m$ with Hopf invariant $H^m_{\psi,\beta}(q): T\to J^m(B)$ {associated to a (pointed) homotopy lifting $\psi \colon L \to J^m_Y(\mathcal{B})$ of the inclusion $L\hookrightarrow Y$ through $q_m$ (as in Definition~\ref{Hopfinv}),} then $\secat_{X\times L \cup K\times Y}(p\times q)\le n+m+1$ with Hopf invariant the composition
\[
\xymatrix{
S\circledast T \ar[rrr]^-{H^n_{\phi,\alpha}(p)\circledast H^m_{\psi,\beta}(q)} &&& J^n(A)\circledast J^m(B) \ar[rr]^-{\Phi_{n,m}^{A,B}} & & J^{n+m+1}(A\times B).
}
\]
\end{thm}

\begin{proof}
We have commutative cubes
\[
\resizebox{1\textwidth}{!}{
\xymatrix{
 & S \ar[rr]^{\alpha} \ar[dl]_{H^n_{\phi,\alpha}(p)} \ar@{^{(}->}'[d][dd] & & K \ar[dd] \ar[dl]_{\phi'} & & T \ar[rr]^{\beta} \ar[ld]_{H^m_{\psi,\beta}(q)} \ar@{^{(}->}'[d][dd] & & L \ar[dd] \ar[ld]_{\psi'} \\
J^n(A)\ar[rr] \ar@{^{(}->}[dd] & & J^n_X({\mathcal A}) \ar[dd] & & J^m(B)\ar[rr] \ar@{^{(}->}[dd] & & J^m_Y({\mathcal B}) \ar[dd] & \\
 & CS \ar'[r][rr] \ar[ld]^>>>>>>>>>>{\;CH^n_{\phi,\alpha}(p)} & & X \ar[ld]_{\sigma} && CT \ar'[r][rr] \ar[ld]^>>>>>>>>>>{\;CH^m_{\psi,\beta}(q)} & & Y \ar[ld]_{\tau}\\
CJ^n(A)\ar[rr] & & J^{n+1}_X({\mathcal A}) && CJ^m(B) \ar[rr] & & J^{m+1}_Y({\mathcal B}) &
}}
\]
constructed as in Proposition \ref{secatupbyone}, where $\sigma$ and $\tau$ are (pointed) homotopy sections of $p_{n+1}$ and $q_{m+1}$ respectively. Applying the naturality statement of Proposition \ref{exteriorjoin}, and splicing the top and right faces of the resulting cube with diagram (\ref{shufflewjoins}), yields a large diagram
\[
\resizebox{1\textwidth}{!}{\xymatrix{
S\circledast T \ar[r]^-{{W}} \ar[d]_{H^n_{\phi,\alpha}(p)\circledast H^m_{\psi,\beta}(q)} & X\times L \cup K \times Y \ar[d] \ar[r]^{{W'}} &  X\times Y \ar[d]_{\sigma\times\tau} \\
J^n(A)\circledast  J^m(B) \ar[r] \ar[d]_{\Phi_{n,m}^{A,B}} & J^{n+1}_X({\mathcal A})\times J_Y^{m}({\mathcal B})\cup J_X^{n}({\mathcal A})\times J_Y^{m+1}({\mathcal B}) \ar[r] \ar[d]_{\Psi_{n,m}^{{\mathcal A},{\mathcal B}}} & J^{n+1}_X({\mathcal A})\times J_Y^{m+1}({\mathcal B}) \ar[d]_{p_{n+1}\times q_{m+1}} \\
J^{n+m+1}(A\times B) \ar[r] & J_{X\times Y}^{n+m+1}({\mathcal A}\times {\mathcal B}) \ar[r]^{(p\times q)_{n+m+1}} & X\times Y.
}}
\]
One easily sees that the middle vertical composition is a (pointed) homotopy lifting of the inclusion $X\times L \cup K \times Y\hookrightarrow X\times Y$ through $(p\times q)_{n+m+1} :J_{X\times Y}^{n+m+1}({\mathcal A}\times {\mathcal B})\to X\times Y$, hence $\secat_{X\times L \cup K\times Y}(p\times q)\le n+m+1$. The left-hand vertical composition is the Hopf invariant associated to this lifting.
\end{proof}

{Using the maps $\bar\chi_m$ defined in Proposition~\ref{TCcatHopf} we now have:}

\begin{cor}\label{forTC2cell}
Let $K$ and $L$ be path-connected spaces, and let
$X=K\cup_\alpha e^{q+1}$ and $Y=L\cup_\beta e^{r+1}$ for maps $\alpha: S^q\to K$ and $\beta: S^r\to L$ {with $q,r\ge1$.} Suppose that $\cat(K)\le n$ and $\cat(L)\le m$ {with $n,m\ge1$,} and let $s: K\to G_n(K)$ and $\sigma: L\to G_m(L)$ be (pointed) sections of the respective Ganea fibrations. Then:
\begin{enumerate}[(a)]
\item $\cat(X\times Y)\le n+m+1$ if the composition
\[
\xymatrix{
S^q\circledast S^r \ar[rr]^-{H_s(\alpha)\circledast H_\sigma(\beta)} && F_n(K)\circledast F_m(L) \ar[r]^-{\Phi_{n,m}^{\Omega K,\Omega L}} & F_{n+m+1}(K\times L)
}
\]
is null-homotopic.
\item $\TC(X)\le 2n+1$ if the composition
\[
\xymatrix{
S^q\circledast S^q \ar[rr]^-{H_s(\alpha)\circledast H_s(\alpha)} && F_n(K)\circledast F_n(K)  \ar[r]^-{\Phi^{\Omega K,\Omega K}_{n,n}} &  F_{2n+1}(K\times K) \ar[r]^>>>>>{\bar\chi_{2n+1}} & F_{2n+1}(K)
}
\]
is null-homotopic. If in addition $K=S^p$ (taking $n=1$) with $p\geq2$ and $q\leq3p-3$, then the vanishing of the above composition is a necessary and sufficient condition for $\TC(X)\leq3$.
\item If $K=S^p$, where $p\ge 2$ and $q\le 3p-3$, then $\TC_{X\times S^p}(X)\le 2$ if and only if the composition
\[
\xymatrix{
S^q\circledast S^{p-1} \ar[rr]^-{H_s(\alpha)\circledast H^0(S^p)} && F_1(S^p)\circledast F_0(S^p) \ar[r]^>>>>{\Phi^{\Omega S^p,\Omega S^p}_{1,0}} & F_2(S^p\times S^p) \ar[r]^>>>>>{\bar\chi_2} & F_2(S^p)
}
\]
is null-homotopic.
\end{enumerate}
\end{cor}

\begin{proof}
\begin{enumerate}[(a)]
\item Naturality of the maps $\Phi_{n,m}$, together with Theorem \ref{Hopfproducts} applied to the product of the fibrations $g_0(X):G_0(X)\to X$ and $g_0(Y):G_0(Y)\to Y$, yield $\cat_{X\times K\cup L\times Y}(X\times Y)\le n+m+1$ with Hopf invariant the lower composition in the diagram
\[\xymatrix{
S^q\circledast S^r \ar[rr]^-{H_s(\alpha)\circledast H_\sigma(\beta)} \ar[drr] && F_n(K)\circledast F_m(L) \ar[d] \ar[rr]^-{\Phi_{n,m}^{\Omega K,\Omega L}} && F_{n+m+1}(K\times L)\ar[d] \\
 && F_n(X)\circledast F_m(Y) \ar[rr]^-{\Phi_{n,m}^{\Omega X,\Omega Y}} && F_{n+m+1}(X\times Y).
}\]
Here the {vertical} maps are induced by inclusions, {whereas the slanted map is $H^n_{G_n(\iota)\circ s,\alpha}(g_0(X))\circledast H^m_{G_m(\iota)\circ \sigma,\beta}(g_0(Y))$}. The result follows.

\item By naturality of the maps $\Phi_{n,n}$ and $\bar\chi_{2n+1}$ we obtain a diagram
\[
\xymatrix{
S^q\circledast S^q \ar[rr]^-{H_s(\alpha)\circledast H_s(\alpha)} \ar[drr] && F_n(K)\circledast F_n(K) \ar[d] \ar[r]^-{\Phi_{n,n}^{\Omega K,\Omega K}} & F_{2n+1}(K\times K)\ar[d]\ar[r]^-{\bar\chi_{2n+1}} & F_{2n+1}(K) \ar[d] \\
 && F_n(X)\circledast F_n(X) \ar[r]^-{\Phi_{n,n}^{\Omega X,\Omega X}} & F_{2n+1}(X\times X) \ar[r]^-{\bar\chi_{2n+1}} & F_{2n+1}(X).
}
\]

The diagram of Proposition \ref{TCcatHopf} shows that composition with $\bar\chi_{2n+1}$ takes Hopf invariants for $\cat(X\times X)$ to Hopf invariants for $\TC(X)$. Therefore $\TC_{X\times K\cup K\times X}(X)\le 2n+1$ with Hopf invariant the lower composition, which yields the first assertion. For the second assertion, a homology argument shows that the map $F_3(S^p)\to F_3(X)$ is a $(3p+q-1)$-equivalence (compare \cite[proof of Lemma~6.26]{CLOT}), so that the vanishing of the upper composition is equivalent to the vanishing of the lower composition. The final conclusion then follows from Proposition~\ref{uniqueness} since the hypothesis $q\leq3p-3$ implies that the Hopf set under consideration is a singleton.

\item {We can safely assume $p\le q$ because if $\alpha$ is null-homotopic, then in fact $H_s(\alpha)$ is null-homotopic and $\TC(X)\le2$. In particular, Proposition~\ref{uniqueness} and the discussion in Example~\ref{ejemploBHH} imply that $\mathcal{H}(\alpha)$ is a singleton. Think of $S^p$ with the cell structure in Example~\ref{cathopfsphere}, so} $X\times S^p\,=\,X\times C^-S^{p-1} \,\cup\, S^p\times S^p\,\cup\, e^{p+q+1}$. The usual deformation retraction of the south hemisphere of $S^p$ to its south pole shows that the inclusion $X\times \ast \,\cup\, S^{p}\times S^p\hookrightarrow X\times C^-S^{p-1} \,\cup\, S^p\times S^p$ is a homotopy equivalence. Thus $\TC_{X\times C^-S^{p-1} \cup S^p\times S^p}(X)= \TC_{X\times \ast \cup S^p\times S^p}(X)\le{\mathrm{cl}(X\times \ast \cup S^p\times S^p)}\le2$. Further, since $\Omega X$ is $(p-2)$-connected and $$\dim(X\times \ast \cup S^p\times S^p)=\max\{q+1,2p\}<3(p-1)+2 = 3p-1,$$ the Hopf set under consideration is a singleton and is given by the lower composition in the diagram
\[
\xymatrix{
S^q\circledast S^{p-1} \ar[rr]^-{H_s(\alpha)\circledast H^0(S^p)} \ar[rrd] && F_1(S^p)\circledast F_0(S^p) \ar[d] \ar[r]^-{\hspace{2mm}\Phi^{\Omega S^p,\Omega S^p}_{1,0}} & F_2(S^p\times S^p)\ar[d] \ar[r]^{\hspace{3mm}\bar\chi_2} & F_2(S^p) \ar[d]\\
 && F_1(X)\circledast F_0(X) \ar[r]^-{\Phi_{1,0}^{\Omega X,\Omega X}} & F_{2}(X\times X) \ar[r]^-{\bar\chi_{2}} & F_{2}(X).
}
\]
This proves the {``if''} statement. A homology argument shows that the map $F_2(S^p)\to F_2(X)$ is a $(2p+q-1)$-equivalence (compare \cite[proof of Lemma~6.26]{CLOT}), and so the lower composition is essential if and only if the upper composition is. This completes the proof.
 \end{enumerate}
 \end{proof}

\section{Application: The topological complexity of $2$-cell complexes}\label{secaptc2cw}

A $2$-cell complex is a finite complex $X = S^p \cup_\alpha e^{q+1}$ presented as the mapping cone of a map of spheres $\alpha: S^{q}\to S^p$, where $q\ge p\ge 1$. In this section we investigate the topological complexity of $2$-cell complexes, using the results of the previous section together with the results of Section~\ref{secsiete} presented in Examples~\ref{degree-2cells}.

\medskip
Recall that the Lusternik--Schnirelmann category of $X$ is determined as follows.
For $p=q$ and
\begin{itemize}
\item $\deg(\alpha)=\pm1$, $X$ is contractible and $\cat(X)=\TC(X)=0$.
\item $\deg(\alpha)=0$, $X\simeq S^p\vee S^{p+1}$ and $\cat(X)=1$ while $\TC(X)=2$.
\item $|\deg(\alpha)|>1$, $\cat(X)=2$ if $p=1$, whereas $\cat(X)=1$ if $p>1$.
\end{itemize}
(The behavior of $\TC(X)$ in the {last} case is discussed below.)
For $p<q$, we can safely assume $p\ge2$---for otherwise $\alpha$ is null-homotopic, in which case $\cat(X)=\cat(S^p\vee S^{q+1})=1$. Then the Berstein--Hilton--Hopf set $\mathcal{H}(\alpha)$ consists of a single element represented by a map $H(\alpha): S^q\to F_1(S^p)$, and we have
\[
\cat(X) = \left\{\begin{array}{ll} 1 & \mbox{if }H(\alpha)=0, \\ 2 & \mbox{if }H(\alpha)\neq 0.
\end{array}\right.
 \]
 Methods for computing $H(\alpha)$ for various $\alpha$ are given in \cite[Chapter 6]{CLOT}.

{As for $\TC(X)$, we start by noticing that $1\le \TC(X)\le 2$ whenever $\cat(X)=1$ (e.g.~if $q>p\ge2$ and $H(\alpha)=0$). Actually,} since $X$ cannot be homotopy equivalent to an odd-dimensional sphere, the main result of \cite{GLO} implies that {in fact} $\TC(X)=2$. This happens whenever $\alpha$ is null-homotopic, or more generally a suspension. Therefore in what follows we will assume we are outside of the stable range, i.e.~we assume $q\ge 2p-1$. Likewise, proof details will be limited to the case $\cat(X)=2$, where $2\le\TC(X)\leq4$.

When {$q=2p-1$} it is possible to give a complete computation of $\TC(X)$ using mainly cohomological arguments. We first address the case $p=1$.

\begin{thm}\label{grados}
Let $X$ be the mapping cone of a map $\alpha: S^1\to S^1$, whose degree we denote by $d_\alpha$. Then
\[
\TC(X)=\left\{\begin{array}{ll} 2 & \mbox{if }d_\alpha=0, \\ 0 & \mbox{if }d_\alpha=\pm 1, \\ 3 & \mbox{if }d_\alpha=\pm 2, \\ 4 & \mbox{otherwise.} \end{array} \right.
\]
\end{thm}
\begin{proof}
The first two cases have been dealt with above. If $d_\alpha=\pm 2$ then $X\simeq \R P^2$, and the result follows from \cite{FTY} since the immersion dimension of $\R P^2$ is $3$.

The remaining {case} can be dealt with using $\TC$-weights of cohomology classes (see \cite{FG}, \cite[Section 4.5]{FarInv}, \cite[Section 4]{FarCosta}). Here $X\simeq M(\Z/k,1)$ is a mod $k$ Moore space, where $k=|d_\alpha|>2$. Let $x\in H^1(X;\Z/k)\cong\Z/k$ and $y\in H^2(X;\Z/k)\cong\Z/k$ be generators. Then $y = \beta(x)$ where $\beta$ is the mod $k$ Bockstein operator. The class
\[
\overline{y} = 1\times y - y\times 1 \in H^2(X\times X;\Z/k)
\]
therefore has $\TC$-weight $2$, by \cite[Theorem 6]{FG}. An easy calculation gives
\[
0\neq \overline{y}^2 = -2 y\times y \in H^4(X\times X;\Z/k)\cong\Z/k,
\]
and so $\TC(X)\ge 2+2=4$ by \cite[Proposition 2]{FG}. Since $\TC(X)\le 2\,\cat(X)\le 4$, this completes the proof.
\end{proof}

In the case of a map $\alpha: S^{2p-1}\to S^p$ with $p\ge 2$, the Berstein--Hilton--Hopf invariant $H(\alpha)\in \pi_{2p-1}(F_1(S^p))$ agrees with the classical Hopf invariant $h(\alpha)\in\Z$ up to a sign. To be more explicit, projection onto the bottom cell of
\[
F_1(S^p) = \Omega S^p\circledast \Omega S^p \simeq \Sigma(\Omega S^p\wedge \Omega S^p)\simeq S^{2p-1}\vee S^{3p-2}\vee\cdots
\]
induces an isomorphism
\[
\pi_{2p-1}(F_1(S^p))\cong \pi_{2p-1}(S^{2p-1})\cong \Z,
\]
which sends $H(\alpha)$ to $\pm h(\alpha)$, see \cite[Section 6.2]{CLOT}.

\begin{thm}\label{ClassicHopf}
Let $X$ be the mapping cone of a map $\alpha: S^{2p-1}\to S^p$ with {$p\geq2$ and} classical Hopf invariant $h(\alpha)\in\Z$. Then
\[
\TC(X) = \left\{\begin{array}{ll} 2 & \mbox{if }h(\alpha)=0, \\ 4 & \mbox{if }h(\alpha)\neq 0. \end{array}\right.
\]
\end{thm}

\begin{proof}
If $h(\alpha)=0$ then $H(\alpha)=0$ and $\TC(X)=2$, as already noted above.

Let $u\in H^p(X;\Z)$ and $v\in H^{2p}(X;\Z)$ be generators of the integral cohomology groups of $X$. The cup product structure is given by $uv=v^2=0$ and $u^2 = h(\alpha) v$. Therefore, if $h(\alpha)\neq 0$ {(so $p$ is even),} we calculate that
\[
0\neq  (1\times u - u\times 1)^4  = 6\, h(\alpha)^2\, v\times v \in H^{4p}(X\times X;\Z)\cong \Z.
\]
Hence $\TC(X)\ge 4$ by the usual zero-divisors cup-length lower bound. Since $\TC(X)\le 4$ this completes the proof.
\end{proof}

In the remainder of the section, and unless it is explicitly noted otherwise, we assume that $\alpha: S^q\to S^p$ is a map of spheres with 
$2p-1< q \le 3p-3$ {(so $q-1>p\geq3$).} Such a map is said to be in the \emph{metastable range}. In this range, projection onto the bottom cell induces an isomorphism
      $
      \pi_q\big(F_1(S^p)\big) \cong \pi_q(S^{2p-1}).
      $
      The image of $H(\alpha)$ under this isomorphism is a map $H_0(\alpha):S^q\to S^{2p-1}$. Note that this map is in the stable range, and is therefore a suspension. {As noted above,} we will assume $H_0(\alpha)\neq 0$, so that $\cat(X)=2$ and $2\le \TC(X)\le 4$.

Using Corollary \ref{forTC2cell}, we give necessary and sufficient conditions for $\TC(X)\le 3$ (Theorem~\ref{le3} below), and sufficient conditions for $\TC(X)\ge 3$ (Theorem~\ref{ge3} below). In many cases we will be able to conclude $\TC(X)=3$. This will be achieved by analyzing some of the Hopf sets associated to the cone decomposition $$\ast=C_0\subset C_1\subset C_2\subset C_3\subset C_4=X\times X$$ given by $C_1=S^p\vee S^p$, $C_2=(X\vee X)\cup(S^p\times S^p)$, and $C_3=(X\times S^p)\cup(S^p\times X)$. Each $C_{i+1}$ is the cone of an obvious map $\alpha_i\colon S_i\to C_i$, where $S_0=S^{p-1}\vee S^{p-1}$, $S_1=S^{q}\vee S^{2p-1}\vee S^{q}$, $S_2=S^{p+q}\vee S^{p+q}$, and $S_3=S^{2q+1}$. In these terms, we write $\mathcal{H}^{n_i}_{i}=\{H_{\phi,\alpha_i}^{n_i}\}_{\phi}$ for the Hopf set arising, as in Definition~\ref{Hopfinv}, from the inequality $\TC_{C_i}(X)\leq n_i$ and all possible homotopy commutative diagrams
\begin{equation}\label{relativeganeas}
\xymatrix{
& & F_{n_i}(X) \ar[d] \\
& & G_{n_i}^{\TC}(X;C_{i+1}) \ar[d] \\
S_i \ar[r]_{\alpha_i} \ar@/^1pc/[uurr]^{H_{\phi,\alpha_i}^{n_i}} & C_i \ar@{^{(}->}[r] \ar@/^1pc/[ur]^-{\phi} & C_{i+1}.}
\end{equation}
Here the vertical fibration is the restriction over the inclusion $C_{i+1}\hookrightarrow X\times X$ of the $n_i$-th $\TC$-Ganea fibration $F_{n_i}(X)\to G_{n_i}^{\TC}(X)\to X\times X$.

\begin{ex}\label{lasdosprimerasobstruccionesdeHopf}
Note that Proposition~\ref{secatupbyone} and the hypothesis $H_0(\alpha)\neq0$ yield
\begin{equation}\label{contensionesfaciles}
2\geq\TC_{C_2}(X)\geq\TC_{X\times \ast}(X)=\cat(X)=2
\end{equation}
and, therefore, $\TC_{C_1}(X)=1$. Proposition~\ref{caracterizacionshida} then implies that the ``first'' two Hopf sets $\mathcal{H}^0_0$ and $\mathcal{H}^1_1$ do not contain the trivial class. Thus, the actual value of $\TC(X)\in\{2,3,4\}$ is determined by the nature of the two ``top'' obstructions $\mathcal{H}^2_2$ and $\mathcal{H}^m_3$ for $m=\TC_{C_3}(X)$. For instance, if $\mathcal{H}^2_2$ is non-trivial, the actual value of $\TC(X)\in\{3,4\}$ is determined by the nature of $\mathcal{H}^3_3$. The latter two Hopf sets are fully described (in the metastable range) by the next results.
\end{ex}

\begin{thm} \label{le3}
Let $X=S^p\cup_\alpha e^{q+1}$, where $\alpha: S^q\to S^p$ is in the metastable range $2p-1< q\le 3p-3$ and $H_0(\alpha)\neq 0$. Then $\mathcal{H}^3_3$ is a singleton and, up to an isomorphism, it consists of the homotopy class $(4+2(-1)^p)H_0(\alpha)\circledast H_0(\alpha)$. Consequently, $\TC(X)\le 3$ if and only if $\hspace{.2mm}(4+2(-1)^p)H_0(\alpha)\circledast H_0(\alpha)=0$.
\end{thm}
\begin{proof}
It has been shown in Corollary~\ref{forTC2cell}(b) that, with the present hypothesis, $\mathcal{H}^3_3$ consists of a single element which, up to an isomorphism, can be identified with the composition
\[
\xymatrix{
S^q \circledast S^q \ar[rr]^-{H(\alpha)\circledast H(\alpha)} && F_1(S^p)\circledast F_1(S^p) \ar[rr]^-{\Phi_{1,1}^{\Omega S^p, \Omega S^p}} && F_3(S^p\times S^p) \ar[r]^{\bar{\chi}_3} & F_3(S^p).
}
\]
Up to homotopy, the first map factors as
\[
\xymatrix{
S^q \circledast S^q \ar[rr]^-{H(\alpha)\circledast H(\alpha)} \ar[rrd]_{H_0(\alpha)\circledast H_0(\alpha)\hspace{7mm}} && F_1(S^p)\circledast F_1(S^p) \\
&& S^{2p-1}\circledast S^{2p-1} \ar@{^{(}->}[u] .
}
\]
By Example \ref{degree-2cells}(a), the degree of the composition $\bar{\chi}_3\circ\Phi_{1,1}^{\Omega S^p, \Omega S^p}$ on the bottom cell $S^{4p-1}$ is $\pm(4+2(-1)^p)$. The result follows since the bottom cell of $F_3(S^p)$ splits off as a wedge summand.
\end{proof}

\begin{rem}\label{paar}
In the situation of Theorem~\ref{le3},~\cite[X-Corollary~8.13]{whiteheadbookelements} gives
$$H_0(\alpha)\circledast H_0(\alpha) = (-1)^{q-2p+1}H_0(\alpha)\circledast H_0(\alpha),$$ which is of order $2$ if $q$ is even. In particular $\mathcal{H}^3_3$ is trivial, and thus $\TC(X)\leq3$, whenever $q$ is even.
\end{rem}

\begin{thm} \label{ge3}
Let $X=S^p\cup_\alpha e^{q+1}$, where $\alpha: S^q\to S^p$ is in the metastable range $2p-1< q\le 3p-3$ and $H_0(\alpha)\neq 0$. Then $\mathcal{H}^2_2$ is a singleton and, up to an isomorphism, it consists of the homotopy class $(2+(-1)^p)H_0(\alpha)$. In particular, $\TC(X)\ge 3$ provided $(2+(-1)^p)H_0(\alpha)\neq0$.
\end{thm}

\begin{rem}\label{trescoincidencias}
Let $C_2'=(X\times\ast)\cup(S^p\times S^p)$ and $C_2''=(\ast\times X)\cup(S^p\times S^p)$, so that $X\times S^p$ and $S^p\times X$ are obtained respectively from $C_2'$ and $C_2''$ by attaching, in each case, a cell of dimension $p+q+1$. Note that $X\times \ast\subset C_2'\subset C_2$ and $\ast\times X\subset C_2''\subset C_2$, so that~(\ref{contensionesfaciles}) yields the equalities $\TC_{C_2}(X)=\TC_{C'_2}(X)=\TC_{C''_2}(X)=2$. A key point in the proof of Theorem~\ref{ge3} (given below) is the observation that such a phenomenon remains valid after attaching the layer of $(p+q+1)$-dimensional cells: $\TC_{(X\times S^p)\cup(S^p\times X)}(X)=\TC_{X\times S^p}(X)=\TC_{S^p\times X}(X)\in\{2,3\}$. Indeed, the three relevant Hopf sets are proven to be singletons, each determined by $(2+(-1)^p)H_0(\alpha)$.
\end{rem}

\begin{proof}[Proof of Theorem~\ref{ge3}]
It has been shown in Corollary~\ref{forTC2cell}(c) that, with the present hypothesis, the Hopf set $\mathcal{H}(1)$ deciding the value of $\TC_{X\times S^p}(X)\in\{2,3\}$ is the singleton determined, up to an isomorphism, by the composition
\[\xymatrix{
S^q \circledast S^{p-1} \ar[rr]^-{H(\alpha)\circledast H^0(S^p) } && F_1(S^p)\circledast F_0(S^p) \ar[rr]^-{\Phi_{1,0}^{\Omega S^p, \Omega S^p}} && F_2(S^p\times S^p) \ar[r]^{\hspace{3.5mm}\bar{\chi}_2} & F_2(S^p).
}\]
As noted in Example~\ref{cathopfsphere}, the Hopf invariant
\[
H^0(S^p):S^{p-1}\to \Omega S^p \simeq {S^{p-1}\cup e^{2p-2}\cup \cdots}
\]
is homotopic to the inclusion of the bottom cell. Therefore the first map in the composition above factors as
\[
\xymatrix{
S^q \circledast S^{p-1} \ar[rr]^-{H(\alpha)\circledast H^0(S^p)\hspace{1.5mm} } \ar[rrd]_{H_0(\alpha)\circledast\mathrm{Id}_{S^{p-1}}\hspace{3mm}} && F_1(S^p)\circledast F_0(S^p) \\
&& S^{2p-1}\circledast S^{p-1} \ar@{^{(}->}[u].
}
\]
The diagonal arrow can be identified with the $p$-th {reduced} suspension ${\widetilde{\Sigma}}^p H_0(\alpha)$, which is essential since $H_0(\alpha)$ is stable.
By Example \ref{degree-2cells}(b), the degree of the composition  $\bar{\chi}_2\circ \Phi_{1,0}^{\Omega S^p, \Omega S^p}$ on the bottom cell $S^{3p-1}$ is $\pm(2+(-1)^p)$. Since the bottom cell of $F_2(S^p)$ splits off as a wedge summand, it follows that, up to an isomorphism, $\mathcal{H}(1)$ is a singleton consisting of the homotopy class $(2+(-1)^p)H_0(\alpha)$.

A similar argument (with the roles of the axes interchanged) gives that the Hopf set $\mathcal{H}(2)$ deciding the value of $\TC_{S^p\times X}(X)\in\{2,3\}$ is the singleton determined, up to an isomorphism, by $(2+(-1)^p)H_0(\alpha)$.

The proof is complete by observing that $\mathcal{H}^2_2$ is a singleton too (by Proposition~\ref{uniqueness} and the metastable range hypothesis), and that the inclusions $X\times S^p\hookrightarrow C_3$ and $S^p\times X\hookrightarrow C_3$ yield maps $\mathcal{H}(1)\to\mathcal{H}^2_2$ and $\mathcal{H}(2)\to\mathcal{H}^2_2$ exhibiting $\mathcal{H}^2_2$ as the cartesian product of $\mathcal{H}(1)$ and $\mathcal{H}(2)$. For instance, the map $\mathcal{H}(2)\to\mathcal{H}^2_2$ arises from the commutative diagram
\[\xymatrix{
& & F_2(X) \ar@{=}[rr] \ar[d] & & F_2(X) \ar@{=}[r] \ar[d] & F_2(X) \ar[d] \\
& & G_{2}^{\TC}(X;S^p\times X) \ar[rr] \ar[d] & &G_{2}^{\TC}(X;C_3) \ar[r] \ar[d] & G_{2}^{\TC}(X) \ar[d] \\
& & S^p\times X \ar@{^{(}->}[rr] & & C_3 \ar@{^{(}->}[r] & X\times X \\
& C_2'' \ar[rr] \ar[ru] \ar@/^1pc/[uur] & & C_2 \ar[ru] \ar@/^1pc/[uur] & \\
S^{p+q} \hspace{5mm} \ar[ur] \ar[rr]^{\iota_2} \ar@/^2pc/[uuuurr] & & S^{p+q}\vee S^{p+q} \ar[ru] \ar@/^2pc/[uuuurr] &
} \]
where the notation is that used in~(\ref{relativeganeas}). Note that the commutativity of the square involving the two short curved liftings follows from the fact that the projection $G_2^{\TC}(X;C_3)\to C_3$ is an equivalence above the dimension of $C_2''$, whereas the commutativity of the square involving the two long curved liftings follows from the fact that the inclusion $F_2(X)\to G_2^{\TC}(X;C_3)$ yields a monomorphism in homotopy groups.
\end{proof}

\begin{ex}
 Let $X=S^3\cup_\alpha e^7$ where $\alpha: S^6\to S^3$ is the Blakers--Massey element, a generator of $\pi_6(S^3)=\Z/12$. Then $0\neq H(\alpha)=H_0(\alpha)\in \pi_6(S^5)=\Z/2$, and the above results imply that $\TC(X)=3$. We remark that the lower bound $\TC(X)\ge 3$ was obtained in \cite[Proposition 30]{Weaksecat} using the weak sectional category, and the upper bound $\TC(X)\le 3$ was obtained in \cite[Example 6]{GC-V} using the fact that $X$ is the $9$-skeleton of the group $Sp(2)$.
\end{ex}

More generally, for $q\in\{2p,2p+1\}$ with $q\leq 3p-3$---i.e.~the first two cases in the metastable range---we have:
\begin{cor}\label{partiuno}
Let $\delta\in\{0,1\}$ and set $p\geq3+\delta$. Then
$$\TC(S^p\cup_\alpha e^{2p+\delta+1})=\begin{cases}
2, & \mbox{ if }H_0(\alpha)=0;\\
3, & \mbox{ if }H_0(\alpha)\neq0.
\end{cases}$$
\end{cor}
\begin{proof}
Just note that multiplication by 3 yields an isomorphism in $\pi_{2p+\delta}(S^{2p-1})=\Z/2$, and that $2 (H_0(\alpha)\circledast H_0(\alpha))=(2H_0(\alpha))\circledast H_0(\alpha)=0$, in view of~\cite[X-Corollary~8.8]{whiteheadbookelements}.
\end{proof}

{Also worth mentioning is:}
\begin{cor}\label{partidos}
{For $p$ odd and $q$ even with $2p-1<q\leq3p-3$,}
$${\TC(S^p\cup_\alpha {e^{q+1}})=}\begin{cases}
{2,} & {\mbox{ if }H_0(\alpha)=0;}\\
{3,} & {\mbox{ if }H_0(\alpha)\neq0.}
\end{cases}$$
\end{cor}

When {$q>3p-3\geq3$} and we are outside of the metastable range, it is still possible to draw conclusions about $\TC(X)$. In particular, the conclusion of Theorem \ref{le3} still holds under the weaker assumption that $H(\alpha)=H_0(\alpha)$ {(and the additional requirement $q<4p-3$ in the case of \ref{le3}(3), which assures the stability of $H_0(\alpha)$ needed in the proof of Theorem~\ref{le3}).} To obtain maps $\alpha: S^q\to S^p$ satisfying {$H(\alpha)=H_0(\alpha)$,} observe that if $\alpha = \gamma\circ\beta$ is a composition
\[
\xymatrix{ S^q\ar[r]^\beta & S^r \ar[r]^\gamma & S^p},
\]
then $H(\alpha) = H(\gamma)\circ \beta$ whenever $H(\beta)=0$ \cite[Proposition 6.18(2)]{CLOT}. Therefore if $H(\beta)=0$ and $\gamma$ is in the stable or metastable range, then $H(\alpha) = H_0(\alpha) = H_0(\gamma)\circ \beta$.

\begin{exs}\label{contieneposiblegap}
\begin{enumerate}[(a)]
\item {The case $q=2p$ fails to lie in the metastable range only for $p=2$ (so $q=4$). The generator of $\pi_4(S^2)=\mathbb{Z}/2$ is represented by the composition $\alpha = \eta\circ \Sigma\eta: S^4\to S^2$ where $\eta: S^3\to S^2$ is the Hopf map.} Then $H(\alpha)=H_0(\alpha)$ is the nonzero element $\Sigma\eta\in \pi_4(S^3)=\Z/2$. For $X=S^2\cup_\alpha e^5$, we conclude as in Theorem \ref{le3} that $\TC(X)\le 3$. {In this case, however, we cannot use Theorem \ref{ge3} to get in fact that $\TC(X)=3$, as the relevant Hopf set fails to be a singleton.}
\item Let $X= S^2\cup_\alpha e^{10}$ where $\alpha = \eta\circ\beta$ and $\beta\in \pi_9(S^3)=\Z/3$ is the generator. This is one of the spaces considered by Iwase in \cite{Iwase}. Here we have $H(\alpha)=H_0(\alpha) = \beta$, and we can conclude as in Theorem \ref{le3} that $\TC(X)\le 3$. We cannot {use} Theorem \ref{ge3} {to conclude} that $\TC(X)\ge 3$ {not only because of the situation noted in item (a) above, but} since {this time} the map $H_0(\alpha)$ is no longer stable (this being the reason why $X$ is a counter-example to Ganea's conjecture). In fact $\Sigma^2\beta = 0$ which, {by the argument in the proof of Theorem~\ref{ge3},} shows that $\TC_{X\times S^2}(X)\le 2$ {(the latter is in fact an equality since $2=\cat(X)=\TC_{X\times\ast}(X)\leq\TC_{X\times S^2}(X)$).}
\end{enumerate}
\end{exs}

\section{Application: Ganea's condition for topological complexity}\label{secapgc4tc}

The analogue of Ganea's conjecture for topological complexity asks whether, for any finite complex $X$ and $k\ge 1$, we have
\begin{equation}\label{GaneaTC}
\TC(X\times S^k) = \TC(X) + \TC(S^k) = \left\{\begin{array}{ll} \TC(X) + 1 & \mbox{if $k$ odd,}\\ \TC(X) + 2 &\mbox{if $k$ even.} \end{array}\right.
\end{equation}
By \cite[Corollary 1.7]{JMP}, the answer is known to be positive when $k\ge 2$ and $X$ is a simply-connected, rational, formal complex of finite type. The answer is also known to be positive, without the formality assumption, if TC is replaced by either of the related rational invariants ${\rm mtc}$ (\cite[Theorem~1.6]{JMP}) or MTC (\cite[Theorem~12]{carrasquel}), the former introduced in~\cite{JMP} and the latter in~\cite{FGKV}. {Theorem~\ref{noGaneaTC} below gives} a counter-example to equation (\ref{GaneaTC}) for all $k\ge 2$ even.

We first observe that the topological complexity of a product of spaces can be described as the sectional category of a product of fibrations.

\begin{lem}\label{TCproduct}
Let $X$ and $Y$ be spaces. Then $\TC(X\times Y) = \secat(\pi_X\times \pi_Y)$.
\end{lem}
\begin{proof}
This follows from the commutative diagram
\[
\xymatrix{
(X\times Y)^I \ar[r]^{T} \ar[d]^{\pi_{X\times Y}} & X^I\times Y^I \ar[d]^{\pi_X\times \pi_Y} \\
X\times Y \times X \times Y \ar[r]^{\tau} & X\times X \times Y \times Y.
}
\]
Here $T$ sends a path in $X\times Y$ to the pair of paths consisting of its projections onto $X$ and $Y$, and $\tau$ transposes the middle two factors. Both of these maps are homeomorphisms, so it follows that $\secat(\pi_{X\times Y}) = \secat(\pi_X\times \pi_Y)$.
\end{proof}

Next we investigate the $\TC$-Hopf invariants for spheres. Consider the cofibration sequence
\[
\xymatrix{
S^{2k-1} \ar[r]^{\beta} & S^k\vee S^k \ar[r]  & S^k\times S^k,
}
\]
where the attaching map $\beta=[\iota_1,\iota_2]:S^{2k-1}\to S^k\vee S^k$ is the Whitehead product of the inclusions of the two wedge factors. Note that the subcomplex $S^k\vee S^k\subseteq S^k\times S^k$ satisfies $\TC_{S^k\vee S^k}(S^k)= 1$ by Proposition \ref{secatprops}. If $k\ge2$, {so that $\hdim(S^k\vee S^k)=k<2k-1$,} then by Proposition \ref{uniqueness} the Hopf set
\[
\mathcal{H}^1_{\beta}(\pi_{S^k})\subseteq \pi_{2k-1}(F_1(S^k))
\]
consists of a single element represented by a map
\[
H^1_\beta(\pi_{S^k}): S^{2k-1}\to F_1(S^k)\simeq S^{2k-1}\vee S^{3k-2}\vee\cdots.
 \]
This map is determined up to homotopy by the degree $d_k\in\Z$ of its projection on to the bottom cell $S^{2k-1}$. The following lemma gives a purely homotopy-theoretic proof of \cite[Theorem 8]{Far}.

\begin{lem}\label{TCHopfspheres}
We have $\pm d_k = 1 + (-1)^k$.
\end{lem}
\begin{proof}
Arguing as in Corollary \ref{forTC2cell}(c), the Hopf invariant $H^1_\beta(\pi_{S^k})$ is given by the composition
\[
\resizebox{1\textwidth}{!}{\xymatrix{
S^{k-1}\circledast S^{k-1} \ar[rr]^-{H^0(S^k)\circledast H^0(S^k)} && F_0(S^k)\circledast F_0(S^k) \ar[rr]^-{\Phi^{\Omega S^k,\Omega S^k}_{0,0}} && F_1(S^k\times S^k) \ar[r]^-{\bar\chi_1} & F_1(S^k)
}}
\]
where $H^0(S^k):S^{k-1}\to \Omega S^k$ can be identified with inclusion of the bottom cell. By Example \ref{degree-2cells}(c) the degree of the map $\bar\chi_1\circ \Phi^{\Omega S^k,\Omega S^k}_{0,0}$ on the bottom cell is $\pm(1+(-1)^k)$, and the result follows.
\end{proof}

\begin{rem}
Hopf set techniques can be used to give a purely homotopy explanation (alternative to that given in~\cite{GC-V}) of the fact that $\TC(S^k)$ agrees with the category of the homotopy cofiber of the diagonal inclusion $S^k\to S^k\times S^k$. The key point in the Hopf set approach comes from the observation that $H^1_\beta(\pi_{S^k})$ agrees up to sign with the (classical) Hopf invariant of the Whitehead square of the identity on $S^k$. This then gives a purely homotopy-theoretic calculation of the topological complexity of spheres. Furthermore, the idea can be applied successfully to 2-cell complexes $S^p\cup e^{q+1}$ satisfying the metastable condition $2p - 1 < q \leq 3(p - 1)$. Details are given in~\cite{GonzalezGrantVandembroucq}.
\end{rem}

\begin{thm}\label{noGaneaTC}
Let $Y$ be the stunted real projective space $\R P^6/\R P^2$, and let $X=Y\vee Y$. Then {$\TC(X)=4$ and, for all $k\ge2$ even, $\TC(X\times S^{k}) = 5$.}
\end{thm}

\begin{proof}
We first show that $\TC(X)=4$. Since $X$ is $6$-dimensional and $2$-connected, the standard dimension/connectivity upper bound gives
\[
\TC(X) \le \frac{2\cdot 6}{3} = 4.
\]
One checks using the cofibration $\R P^2\to \R P^6 \to Y$ that {there is a ring monomorphism} $$\Z/2[u]/(u^3)\hookrightarrow H^*(Y;\Z/2)$$ where $|u|=3$. Therefore the cup-length of $H^*(Y\times Y;\Z/2)$ is $4$, and we have
\[
4\le \cat(Y\times Y)\le \TC(Y\vee Y) = \TC(X)
\]
on applying the result of \cite[Theorem 3.6]{Dran} or \cite[Corollary 2.9]{GLO2}. (Alternatively, the inequality $\TC(X)\ge 4$ follows directly from a zero-divisors calculation in mod $2$ cohomology.) This completes the proof of the first claim.

The lower bound $\TC(X\times S^k)\ge5$ is a quick calculation using zero-divisors in mod $2$ cohomology, and is omitted. We prove that $\TC(X\times S^k)\le 5$ using Hopf invariants.

By Lemma \ref{TCproduct} we have $\TC(X\times S^{k}) = \secat(\pi_X\times \pi_{S^k})$. Let $K$ denote the $11$-skeleton of the standard product cell structure on $X\times X$. Then $X\times X \simeq K\cup_\alpha CS$ where $S = \vee_{i=1}^4 S^{11}$ is a wedge of spheres and $\alpha: S\to K$ is the attaching map of the four $12$-cells. Similarly we have $S^k\times S^k = L\cup_{\beta} e^{2k}$, where $L = S^k\vee S^k$ and $\beta: S^{2k-1}\to L$ is the attaching map of the top cell.

By Proposition \ref{secatprops} we have $\TC_K(X)\le \frac{11}{3}<4$, and since relative topological complexity can increase by at most one on attaching the top cells, $\TC_K(X)=3$. Let $\phi: K\to G_3^{\TC}(X)$ be any (pointed) lift of the inclusion $K\hookrightarrow X\times X$ through $g_3^{\TC}(X)$. The associated Hopf invariant $H^3_{\phi,\alpha}(\pi_X): S\to F_3(X)$ is non-trivial, \red{since otherwise we would have $\TC(X)\leq3$, and is} of order $2$, since \red{the group in which it lies is entirely 2-torsion:}
\begin{align*}
[S,F_3(X)] & \;\cong\; \bigoplus \pi_{11}(F_3(X))
           \;\cong\; \bigoplus H_{11}(F_3(X))
           \;\cong\; \bigoplus H_{8}((\Omega X)^{\wedge 4})\\
           & \;\cong\; \bigoplus H_2(\Omega X)^{\otimes 4}
           \;\cong\; \bigoplus H_3(X)^{\otimes 4}
           \;\cong\; \bigoplus \Z/2.
\end{align*}

By Theorem \ref{Hopfproducts} we have
$
\secat_{X\times X\times L \cup K\times S^{k}\times S^{k}}(\pi_X\times \pi_{S^k}) \le 5
$
with an element of the Hopf set $\mathcal{H}_{W}^5(\pi_X\times \pi_{S^k})$ given by the composition
\[
\xymatrix{
S\circledast S^{2k-1} \ar[rrr]^-{H^3_{\phi,\alpha}(\pi_X)\circledast H^1_\beta(\pi_{S^k})} &&& F_3(X)\circledast F_1(S^{k}) \ar[rr]^-{\Phi^{\Omega X,\Omega S^k}_{3,1}} && F_5(X\times S^{k}).
 }
 \]
 The first map in this composition is null-homotopic. To see this, note that by Lemma \ref{TCHopfspheres} it factors as
 \[
 \xymatrix{
 S\circledast S^{2k-1} \ar[rrr]^-{H^3_{\phi,\alpha}(\pi_X)\circledast H^1_\beta(\pi_{S^k})} \ar[rrrd]_{H^3_{\phi,\alpha}(\pi_X)\circledast (\pm 2)} &&& F_3(X)\circledast F_1(S^{k})\\
 &&& F_3(X)\circledast S^{2k-1} \ar@{^{(}->}[u]
 }
 \]
 where the diagonal map is the join of a map of order $2$ with a map of even degree. Therefore $0\in \mathcal{H}_{W}^5(\pi_X\times \pi_{S^k})$ and $\TC(X\times S^k)\le 5$, as claimed.
\end{proof}

\section{Topological shuffle maps}
\label{secsiete}

Let $p:{\mathcal A}\to X$ and $q:{\mathcal B}\to Y$ be fibrations with respective fibres $A$ and $B$. In this section we construct, for any non negative integers $n$ and $m$, a commutative diagram of the following form:
$$
\xymatrix{
J^{n+1}(A)\times J^{m}(B)\cup J^{n}(A)\times J^{m+1}(B)\ar[rr]^-{{\Psi_{n,m}^{A,B}}} \ar[d]&& J^{n+m+1}(A\times B)\ar[d] \\
J^{n+1}_X({\mathcal A})\times J_Y^{m}({\mathcal B})\cup J_X^{n}({\mathcal A})\times J_Y^{m+1}({\mathcal B})\ar[rr]^-{{\Psi_{n,m}^{{\mathcal A},{\mathcal B}}}} \ar[d]&& J_{X\times Y}^{n+m+1}({\mathcal A}\times {\mathcal B})\ar[d] \\
X\times Y \ar@{=}[rr] &&X\times Y.}
$$
The construction of $\Psi_{n,m}^{{\mathcal A},{\mathcal B}}$ is given in Subsection \ref{shufflemap} 
based on the standard decomposition of a product of simplices into simplices, which is recalled in Subsection \ref{standarddecomp}.
{The needed diagram~(\ref{shufflewjoins}) in Subsection~\ref{HopfProducts} is then obtained by setting $\Phi_{n,m}^{A,B}:=\Psi_{n,m}^{A,B}\circ \xi$ where $$\xi: J^ n(A)\circledast J^m(B) \to J^{n+1}(A)\times J^{m}(B)\cup J^{n}(A)\times J^{m+1}(B)$$
is induced by the maps $\kappa_n\colon CJ^n(A)\to J^{n+1}(A)$ and $\kappa_n\colon CJ^m(B)\to J^{m+1}(B)$ introduced in Remark~\ref{fibredjoins}.}
{In Subsection~\ref{homotopycharacterization} we characterize the homotopy type of $\Phi_{n,m}^{A,B}$ in terms of $(n,m)$ shuffles, and}
in Subsection \ref{homologyComp} we compute the morphism induced by $\Phi_{n,m}^{A,B}$ in homology and prove Theorem~\ref{eli}.

\subsection{Standard decomposition of the product $\Delta^n\times \Delta^m$}
\label{standarddecomp}

We recall here the standard subdivision of the product of two simplices into simplices as described in \cite[p.68]{EilenbergSteenrod}, giving a different presentation which will be more convenient for our future
 computations in homology (see also \cite[pp. 277-278]{Hatcher}).

Let $n$ and $m$ two non-negative integers. Let us denote by $\mathcal{S}_{n+m}$ the set of permutations of $\{1,\ldots, n+m\}$ and by $\mathcal{S}_{n,m}\subset \mathcal{S}_{n+m}$ the subset of $(n,m)$ shuffles. A shuffle $\theta\in  \mathcal{S}_{n,m}$ will be indicated by
$\theta=(\theta(1)\cdots \theta(n)\,||\,\theta(n+1) \cdots \theta(n+m))$
{(so $1\leq\theta(1)<\cdots<\theta(n)\leq n+m$, $1\leq\theta(n+1)<\cdots<\theta(n+m)\leq n+m$, and $\{\theta(1),\ldots,\theta(n)\}\cap\{\theta(n+1),\ldots,\theta(n+m)\}=\varnothing$)} and can be represented in a $n\times m$ rectangular grid with vertices $(i,j)$ ($0\leq i\leq n$, $0\leq j\leq m$) by a continuous path going from $(0,0)$ to $(n,m)$ formed by $n$ horizontal edges and $m$ vertical edges. More precisely, the $k$-th edge of the edgepath $\theta$ is horizontal if $1\leq \theta^{-1}(k)\leq n$ and vertical if $n+1\leq \theta^{-1}(k)\leq n+m$. For instance, the edgepath representing the $(5,3)$ shuffle $(1\,2\,4\,6\,7\,||\,3\,5\,8)$ is shown in Figure 1 below. The signature of the shuffle $\theta$ is denoted by $(-1)^{|\theta|}$ where $|\theta|$ can be interpreted as the number of squares in the grid lying below the path $\theta$.

\vspace{-.45cm}
\psset{xunit=1.0cm,yunit=1.0cm}
\begin{pspicture}(-3,-1)(7.5,4)
\psgrid[subgriddiv=0,gridlabels=0,gridcolor=lightgray](0,0)(0,0)(5,3)
\psline(0,0)(1,0)
\psline(1,0)(2,0)
\psline(2,0)(2,1)
\psline(2,1)(3,1)
\psline(3,1)(3,2)
\psline(3,2)(5,2)
\psline(5,2)(5,3)
\begin{scriptsize}
\psdots[dotstyle=*](0,0)
\psdots[dotstyle=*](1,0)
\psdots[dotstyle=*](2,0)
\psdots[dotstyle=*](2,1)
\psdots[dotstyle=*](3,1)
\psdots[dotstyle=*](3,2)
\psdots[dotstyle=*](4,2)
\psdots[dotstyle=*](5,2)
\psdots[dotstyle=*](5,3)
\rput[bl](-1,-0.2){$(0,0)$}
\rput[bl](5.5,3){$(n,m)$}
\rput[bl](-.7,-.8){\normalsize Figure 1. The $(5,3)$ shuffle $(1\,2\,4\,6\,7\,||\,3\,5\,8)$}
\end{scriptsize}
\end{pspicture}

\vspace{2mm}
The product of the standard simplices $\Delta^n=[e_0,\ldots ,e_{n}]$ and $\Delta^m=[e'_0,\ldots ,e'_{m}]$
(where $(e_i)_{0\leq i\leq n}$ and  $(e'_j)_{0\leq j\leq m}$ are the canonical bases of $\R^ {n+1}$ and $\R^ {m+1}$)
 is the union of $\binom{n+m}{n}$ simplices of dimension $(n+m)$, each of which is determined by a $(n,m)$ shuffle. Explicitly, let $\theta$ be a $(n,m)$ shuffle.
Then the (ordered) $(n+m)$-simplex $\Delta_{\theta}\subset \Delta^ n\times \Delta^m$ associated to $\theta$ is given by
$\Delta_{\theta}=[v_1,\ldots , v_{n+m+1}]$
with $v_k=(e_{i_k},e'_{j_k})$ where $(i_k,j_k)$ is the $k$-th vertex of the edgepath $\theta$.
Observe that $v_1=(e_0,e_0')$ and $v_{n+m+1}=(e_{n},e_{m}')$ for any $\theta$.



 Taken together, the simplices $\Delta_{\theta}$ form a chain
 \begin{equation}\label{lachain}
 \sum\limits_{\theta\in {\mathcal S}_{n,m}}(-1)^{|\theta|}\Delta_{\theta}
\end{equation}

\noindent which represents a generator $\iota_{n,m}$ of
 \[
 H_{n+m}(\Delta^n\times \Delta^m, \partial(\Delta^n \times \Delta^m);\Z)\cong\tilde{H}_{n+m}((\Delta^n\times \Delta^m)/ \partial(\Delta^n \times \Delta^m);\Z)\cong\Z.
 \]


We denote by $\psi_{\theta}:\Delta_{\theta}\to \Delta^{n+m}$ the affine map that
sends the $k$-th vertex of $\Delta_{\theta}$ to the $k$-th vertex of $\Delta^{n+m}$. This gives a homeomorphism from $(\Delta_{\theta},\partial\Delta_{\theta})$ to $(\Delta^{n+m},\partial\Delta^{n+m})$ which induces a map $$\bar{\psi}_{\theta}:\Delta_{\theta}/\partial\Delta_{\theta}\to\Delta^{n+m}/\partial\Delta^{n+m}$$
of degree $1$.

When $\theta$ runs over the set ${\mathcal S}_{n,m}$ of $(n,m)$ shuffles, the maps $\psi_{\theta}$ glue together. In order to see that, it suffices to check that $\psi_{\theta}$ and $\psi_{\theta'}$ agree on $\Delta_{\theta}\cap \Delta_{\theta'}$ when this intersection is a simplex of dimension exactly $n+m-1$. This happens when the edgepaths $\theta$ and $\theta'$ differ in only one vertex, say the $k$-th vertex. In that case $2\leq k\leq n+m$,  {and $\Delta_{\theta}\cap \Delta_{\theta'}$ is the $(n+m-1)$-simplex determined by the (ordered) common vertices in the edge paths of $\theta$ and $\theta'$. But those vertices are sent preserving order, both by $\psi_{\theta}$ and $\psi_{\theta'}$, into the ordered vertices of the face of $\Delta^{n+m}$ opposite to its $k$-th vertex. (In the situation above, $|\theta|$ and $|\theta'|$ differ by one, which explains why~(\ref{lachain}) is a relative cycle, obviously representing a generator $\iota_{n,m}$.)}
As a consequence, the maps $\psi_{\theta}$ produce
a map
\begin{equation}\label{primerfinm}
\psi_{n,m}:\Delta^n \times \Delta^m \to \Delta^{n+m}
\end{equation}
{which} sends the boundary of each $\Delta_{\theta}$ {and, in particular,} the boundary of $\Delta^n\times \Delta^m$ to the boundary of $\Delta^{n+m}$.

When passing to the quotient of $\Delta^ n\times \Delta^m$ first by the boundary $\partial(\Delta^n \times \Delta^m)$ and secondly by the boundaries of the simplices $\Delta_{\theta}$ we get the \textit{shuffle pinch map}
$$\nu_{n,m}:(\Delta^ n\times \Delta^m)/\partial(\Delta^n \times \Delta^m)\to \bigvee\limits_{\theta\in {\mathcal S}_{n,m}}\Delta_{\theta}/\partial \Delta_{\theta}$$
which gives the following morphism in homology:
\begin{equation}\label{inducidoabajo}
\iota_{n,m}\mapsto \bigoplus\limits_{\theta\in {\mathcal S}_{n,m}}(-1)^{|\theta|}\iota_{n+m}^{\theta}.
\end{equation}
Here, as before, $\iota_{n,m}$ is the generator of $H_{n+m}(\Delta^n\times \Delta^m/ \partial(\Delta^n \times \Delta^m);\Z)$ corresponding to the chain~(\ref{lachain}), and $\iota_{n+m}^{\theta}\in H_{n+m}(\Delta_{\theta}/\partial \Delta_{\theta};\Z)=\Z$ is the generator corresponding to the ordered simplex $\Delta_{\theta}$.

The map $\psi_{n,m}:\Delta^n \times \Delta^m \to \Delta^{n+m}$ then fits into the commutative diagram
\begin{equation}\label{diagramPinchmap}
\xymatrix{
\Delta^n\times \Delta^m \ar[rr]^{\psi_{n,m}} \ar[d]&& \Delta^{n+m} \ar[d]\\
\disfrac{\Delta^n\times \Delta^m}{\partial (\Delta^n\times \Delta^m)}\ar[r]_{\nu_{n,m}} &\bigvee\limits_{\theta\in {\mathcal S}_{n,m}}\Delta_{\theta}/\partial \Delta_{\theta}\ar[r]^-{(\bar{\psi}_{\theta})}&\disfrac{\Delta^{n+m}}{\partial \Delta^{n+m}}.
}
\end{equation}
Since the maps $\bar{\psi}_{\theta}$ are of degree $1$, the bottom line induces in homology the morphism $$\iota_{n,m}\mapsto \sum\limits_{\theta\in {\mathcal S}_{n,m}}(-1)^{|\theta|}\iota_{n+m}
$$ where $\iota_{n+m}$ corresponds to the standard generator of $H_{n+m}(\Delta^{n+m}/\partial \Delta^{n+m};\Z)$.

\subsection{The maps $\Psi_{n,m}^{{\mathcal A},{\mathcal B}}$}
\label{shufflemap}

For any integer $k$ we shall denote by $T_k$ the map
$$\begin{array}{rcl}
T_k: {\mathcal A}^k_X\times {\mathcal B}^k_Y & \to & ({\mathcal A}\times {\mathcal B})^k_{X\times Y}\\
\big((a_1,\ldots,a_k), (b_1,\ldots,b_k)\big)&\mapsto & \big((a_1,b_1),\ldots,(a_k,b_k)\big),
\end{array}$$
and by $\Delta_k^{\mathcal A}:{\mathcal A}\to {\mathcal A}^k_X$ the $k$-th diagonal, $a\mapsto (a,a,\ldots,a)$. In addition, for an $(n,m)$ shuffle $\theta$, we shall define two sequences of integers $\alpha_0,\ldots ,\alpha_n$ and $\beta_0,\ldots ,\beta_m$.
In the grid illustrated in Figure~1, $\alpha_i$ will correspond to the number of vertices of the edgepath $\theta$ belonging to the column $i$, while $\beta_j$ will correspond to the number of vertices of the edgepath $\theta$ belonging to the row $j$. Explicitly, if $n\neq 0$, we set
$$
\begin{array}{lr}
\alpha_0=\theta(1),\\
\alpha_i=\theta(i+1)- \theta(i),& 1\leq i\leq n-1,\\
\alpha_{n}=n+m+1-\theta(n),
\end{array}
$$
and, if $n=0$, we set $\alpha_0=m+1$. Similarly, if $m\neq 0$, we set
$$
\begin{array}{lr}
\beta_0=\theta(n+1),\\
\beta_j=\theta(n+j+1)-\theta(n+j),& 1\leq j\leq m-1,\\
\beta_{m}=n+m+1-\theta(n+m),
\end{array}
$$
and, if $m=0$, we set $\beta_0=n+1$.

We thus define the map
\[
\delta_{\theta}:{\mathcal A}_X^{n+1}\times {\mathcal B}_Y^{m+1}\to ({\mathcal A}\times {\mathcal B})_{X\times Y}^{n+m+1}
\]
by $\delta_{\theta}=T_{n+m+1}\circ\big((\Delta^{\mathcal A}_{\alpha_0}\times \cdots \times \Delta^{\mathcal A}_{\alpha_{n}})\times (\Delta^{\mathcal B}_{\beta_0}\times \cdots \times \Delta^{\mathcal B}_{\beta_{m}})\big)$.

If two edgepaths $\theta$ and $\theta'$ differ in only one vertex, say the $k$-th vertex, the maps $\delta_{\theta}$ and $\delta_{\theta'}$ differ in only their $k$-th component. Since, in that case, the $k$-th {(barycentric)} components of $\psi_{\theta}$ and $\psi_{\theta'}$ vanish on $\Delta_{\theta}\cap \Delta_{\theta'}$ we obtain that the maps
\[
\xymatrix{
{\mathcal A}_X^{n+1}\times {\mathcal B}_Y^{m+1}\times \Delta_{\theta}\ar[r]^-{\delta_{\theta}\times \psi_{\theta}} &  ({\mathcal A}\times {\mathcal B})_{X\times Y}^{n+m+1}\times  \Delta^{n+m}\ar[r] &  J^{n+m}_{X\times Y}({\mathcal A}\times {\mathcal B})
}
\]
glue together to give a map
$${\mathcal A}^{n+1}_X\times {\mathcal B}^{m+1}_Y\times \Delta^ n\times \Delta^m\to J^{n+m}_{X\times Y}({\mathcal A}\times {\mathcal B}).$$
This map induces
$$\begin{array}{rcl}
\psi_{n,m}^{{\mathcal A},{\mathcal B}}:J^n_X({\mathcal A})\times J^m_Y({\mathcal B})& \to& J^{n+m}_{X\times Y}({\mathcal A}\times {\mathcal B})\\
\left(\langle \mathbf{a} \mid \mathbf{t}\rangle,\langle \mathbf{b} \mid \mathbf{s}\rangle\right)& \mapsto &
\langle \delta_{\theta}(\mathbf{a},\mathbf{b})\mid \psi_{\theta}(\mathbf{t},\mathbf{s})\rangle, \quad (\mathbf{t},\mathbf{s})\in \Delta_{\theta}.
\end{array}$$
The map $\psi_{n,m}^{{\mathcal A},{\mathcal B}}$ is over $X\times Y$ and {restricts to} a map between fibres
\begin{equation}\label{segundafinm}
\psi_{n,m}^{{ A},{ B}}:J^n({ A})\times J^m({ B}) \to J^{n+m}({A}\times { B})
\end{equation}
which is given by the construction above for trivial fibrations $A\to *$ and $B\to *$. That is, $\psi_{n,m}^{{ A},{ B}}$ comes from the gluing of maps $\delta_{\theta}\times \psi_{\theta}$ where $\delta_{\theta}:{ A}^{n+1}\times { B}^{m+1}\to  ({ A}\times { B})^{n+m+1}$ is given by
$$\delta_{\theta}=T_{n+m+1}\circ((\Delta^{A}_{\alpha_0}\times \cdots \times \Delta^{ A}_{\alpha_{n}})\times (\Delta^{B}_{\beta_0}\times \cdots \times \Delta^{B}_{\beta_{m}})).$$
{In turn,~(\ref{segundafinm}) recovers~(\ref{primerfinm}) for $A=B=\ast$.}

\begin{rem} We notice that the following commutative diagram produced by the map $\psi_{n,m}^{{\mathcal A},{\mathcal B}}$
$$\xymatrix{
J^n_X({\mathcal A})\times J^m_Y({\mathcal B})\ar[rr]^{\psi_{n,m}^{{\mathcal A},{\mathcal B}}}\ar[rd] && J^{n+m}_{X\times Y}({\mathcal A}\times {\mathcal B})\ar[ld]\\
&X\times Y
}$$
permits us to recover the well-known inequality $\secat(p\times q)\leq \secat(p)+\secat(q)$, which is classically established using the open cover definition of the sectional category.
\end{rem}

\begin{lem} The following diagram
$$\xymatrix{
J_X^{n+1}({\mathcal A}) \times J_Y^{m}({\mathcal B}) \ar[ddr]_{\psi_{n+1,m}^{{\mathcal A},{\mathcal B}}} &
J_X^{n}({\mathcal A})\times J_Y^{m}({\mathcal B})\ar[l]_{\hspace{3mm}\jmath_n\times \mathrm{Id}}\ar[r]^{\mathrm{Id}\times \jmath_m\hspace{3mm}}
\ar[d]^{{\psi_{n,m}^{{\mathcal A},{\mathcal B}}}}
& J_X^{n}({\mathcal A}) \times J_Y^{m+1}({\mathcal B})\ar[ddl]^{\psi_{n,m+1}^{{\mathcal A},{\mathcal B}}}\\ &
{J^{n+m}_{X\times Y}({\mathcal A}\times {\mathcal B})}
\ar[d]|-{\raisebox{1.5mm}{\scriptsize${\rule{0mm}{2mm}\jmath_{n+m}}$}} & \\
& J^{n+m+1}_{X\times Y}({\mathcal A}\times {\mathcal B})&
}$$
is commutative.
\end{lem}

\begin{proof} Let $\langle \mathbf{a} \mid \mathbf{t}\rangle \in J_X^{n}({\mathcal A})$ and $\langle \mathbf{b} \mid \mathbf{s}\rangle \in J_Y^{m}({\mathcal B})$ with $(\mathbf{t},\mathbf{s})\in \Delta_{\theta}\subset \Delta^n\times \Delta^m$ and let $(\alpha_i)$ and $(\beta_j)$ the sequences of integers associated to the $(n,m)$ shuffle $\theta$.

Recall from Remark~\ref{fibredjoins} that
$$(\jmath_n\times \mathrm{Id})(\langle \mathbf{a} \mid \mathbf{t}\rangle,\langle \mathbf{b} \mid \mathbf{s}\rangle)=(\langle \mathbf{a},a_{n+1} \mid \mathbf{t},0\rangle,\langle \mathbf{b} \mid \mathbf{s}\rangle).$$
The inclusion $\Delta^n\times \Delta^m\hookrightarrow \Delta^{n+1}\times \Delta^m$, $(\mathbf{t},\mathbf{s})\mapsto
(\mathbf{t},0,\mathbf{s})$, restricts to the inclusion of $\Delta_{\theta}$ into $\Delta_{\theta'}\subset \Delta^{n+1}\times \Delta^m$ where $\theta'$ is the  $(n+1,m)$ shuffle given by
$$(\theta(1)\cdots \theta(n) \,\,\,n+m+1\,||\,\theta(n+1) \cdots \theta(n+m)).$$
Through this inclusion $\Delta_{\theta}$ can be seen as the face of $\Delta_{\theta'}$ opposite to its vertex $(e_{n+1},e'_m)$. The  sequences $(\alpha'_i)$ and $(\beta'_j)$ associated to $\theta'$ are given by
$$\alpha_0,\ldots ,\alpha_n,1 \qquad \beta_0,\ldots ,\beta_m+1$$
As a consequence we obtain
$${\psi_{n+1,m}^{{\mathcal A},{\mathcal B}}}(\jmath_n\times \mathrm{Id})(\langle \mathbf{a} \mid \mathbf{t}\rangle,\langle \mathbf{b} \mid \mathbf{s}\rangle)=\langle \delta_{\theta}(\mathbf{a},\mathbf{b}),(a_{n+1},b_m) \mid \psi_{\theta}(\mathbf{t},\mathbf{s}),0\rangle.$$
Similarly we see
$${\psi_{n,m+1}^{{\mathcal A},{\mathcal B}}}(\mathrm{Id}\times \jmath_m)(\langle \mathbf{a} \mid \mathbf{t}\rangle,\langle \mathbf{b} \mid \mathbf{s}\rangle)=\langle \delta_{\theta}(\mathbf{a},\mathbf{b}),(a_{n},b_{m+1}) \mid \psi_{\theta}(\mathbf{t},\mathbf{s}),0\rangle$$
which gives the result.
\end{proof}

As a consequence the maps $\psi^ {{\mathcal A},{\mathcal B}}_{n+1,m}$ and $\psi^ {{\mathcal A},{\mathcal B}}_{n,m+1}$ glue together and give a map
$$\xymatrix{
 \Psi_{n,m}^{{\mathcal A},{\mathcal B}}: J_X^{n+1}({\mathcal A})\times J_Y^{m}({\mathcal B})
\cup J_X^{n}({\mathcal A})\times J_Y^{m+1}({\mathcal B}) \ar[r]& J^{n+m+1}_{X\times Y}({\mathcal A}\times {\mathcal B})
}$$
which is over $X\times Y$. As before this map induces at the level of fibres a map
$$\xymatrix{
 \Psi_{n,m}^{{A},{B}}:J^{n+1}({ A})\times J^{m}({ B})
\cup J^{n}({A})\times J^{m+1}({B}) \ar[r]& J^{n+m+1}({ A}\times { B}).
}$$

\subsection{The maps $\Phi_{n,m}^{A,B}$}
\label{homotopycharacterization}

For any $k$, the inclusion $J^k(A) \to J^{k+1}(A)$ factors as
$$\xymatrix{
J^k(A)\ar[rr]\ar[rd] && J^{k+1}(A)\\
& CJ^{k}(A)\ar[ru]
}$$
where the map $CJ^{k}(A)\to J^{k+1}(A)$ sends $[\langle \mathbf{a} \mid \mathbf{t}\rangle,u]$ to $\langle \mathbf{a},\ast \mid (1-u)\mathbf{t},u\rangle$ (see Remark~\ref{fibredjoins}). With these maps the following diagram is commutative
$$\xymatrix{
CJ^n(A) \times J^{m}(B) \ar[d]& J^n(A) \times J^{m}(B)\ar[l]\ar[r]\ar@{=}[d] & J^n(A) \times CJ^{m}(B)\ar[d]\\
J^{n+1}(A) \times J^{m}(B)  & J^n(A) \times J^{m}(B)\ar[l]\ar[r] &J^n(A) \times J^{m+1}(B)
}$$
and induces a map $\xi: J^ n(A)\circledast J^ m(B)\to J^{n+1}(A)\times J^{m}(B)
\cup J^{n}(A)\times J^{m+1}(B)$. The composition $\Psi_{n,m}^{A,B}\circ \xi$ is a map
$$
\Phi_{n,m}^{A,B}: J^n(A)\circledast J^m(B)\to  J^{n+m+1}(A\times B)
$$
which we call {the} \emph{topological {$(n,m)$} shuffle map {for $(A,B)$}}. Note that by construction this map is natural in $A$ and $B$, {and renders the commutative diagram~(\ref{shufflewjoins}) which is the key ingredient in the proof of Theorem~\ref{Hopfproducts} for identifying Hopf invariants of products of fibrations.}

Items (a) and (b) of Remarks~\ref{rmkpinchmap} and~(\ref{lasidentificacionesobvias}) imply that $\Phi_{n,m}^{A,B}$ is a map between spaces each of which has the homotopy type of a $(n+m+1)$-fold suspension. In such terms, $\Phi_{n,m}^{A,B}$ turns out to be a sum of suspended maps which we describe next.

{Start by noticing that} $\Phi_{n,m}^{A,B}$ is given on $J^n(A)\times C J^m(B)$ by
$$
(\langle\mathbf{a}\!\mid\mathbf{t}\rangle,[\langle \mathbf{b}\!\mid\mathbf{s}\rangle,u]) \mapsto
\langle \delta_{\theta}(\mathbf{a},(\mathbf{b},\ast))\mid \psi_{\theta}(\mathbf{t},(1-u)\mathbf{s},u)\rangle
$$
when $(\mathbf{t},(1-u)\mathbf{s},u)\in \Delta_{\theta}, \theta \in {\mathcal S}_{n,m+1}$, and its expression on $CJ^n(A)\times J^m(B)$ is
$$
([\langle\mathbf{a}\!\mid\mathbf{t}\rangle,u],\langle \mathbf{b}\!\mid\mathbf{s}\rangle) \mapsto
\langle \delta_{\theta}((\mathbf{a},\ast),\mathbf{b})\mid \psi_{\theta}((1-u)\mathbf{t},u,\mathbf{s})\rangle
$$
when $((1-u)\mathbf{t},u,\mathbf{s})\in \Delta_{\theta}, \theta \in {\mathcal S}_{n+1,m}$. {In particular, for the identification map}
$$
r:J^{n+m+1}(A\times B)\to \widetilde{\Sigma}^{n+m+1}(A\times B)^{n+m+2}=\left((A\times B)^{n+m+2}\right)\wedge \disfrac{\Delta^{n+m+1}}{\partial(\Delta^{n+m+1})}
$$
in~(\ref{lasidentificacionesobvias}), the composition $r\circ\Phi^{A,B}_{n,m}$ is (based-)constant on $J^n(A)\times J^m(B)=J^n(A)\times CJ^m(B) \hspace{1mm}\cap\hspace{1mm} CJ^n(A)\times J^m(B)$ as well as on all points $(\ast,[\ast,u])\in J^n(A)\times CJ^m(B)$ and $([\ast,u],\ast)\in CJ^n(A)\times J^m(B)$ for $u\in I$, where base points are as in~(\ref{puntobasedeljoin}). Consequently $r\circ\Phi^{A,B}_{n,m}$ factors through the difference pinch map:
$$\xymatrix{
J^n(A)\circledast J^m(B) \ar[r]^{\Phi^{A,B}_{n,m}\;\,}\ar[d]_-{\nu} &
J^{n+m+1}(A\times B) \ar[d]^-r \\
\widetilde{\Sigma}\left(J^n(A)\times J^m(B)\right) \vee \widetilde{\Sigma}\left(J^n(A)\times J^m(B)\right) \ar[r]^-{(\Phi',\Phi'')\;} & \left((A\times B)^{n+m+2}\right)\wedge \disfrac{\Delta^{n+m+1}}{\partial(\Delta^{n+m+1})}.
}$$
The (co)components $\Phi'$ and $\Phi''$ are given by
$$
\Phi'\left(\left[(\langle\mathbf{a}\mid\mathbf{t}\rangle,\langle\mathbf{b}\mid\mathbf{s}\rangle),u\rule{0mm}{3mm}\right]\rule{0mm}{3.5mm}\right)=\delta_{\theta}(\mathbf{a},(\mathbf{b},\ast)) \wedge \psi_{\theta}(\mathbf{t},(1-u)\mathbf{s},u)
$$
when $(\mathbf{t},(1-u)\mathbf{s},u)\in \Delta_{\theta}, \theta \in {\mathcal S}_{n,m+1}$, and by
$$
\Phi''\left(\left[(\langle\mathbf{a}\mid\mathbf{t}\rangle,\langle\mathbf{b}\mid\mathbf{s}\rangle),u\rule{0mm}{3mm}\right]\rule{0mm}{3.5mm}\right)=\delta_{\theta}((\mathbf{a},\ast),\mathbf{b}) \wedge \psi_{\theta}((1-u)\mathbf{t},u,\mathbf{s})
$$
when $((1-u)\mathbf{t},u,\mathbf{s})\in \Delta_{\theta}, \theta \in {\mathcal S}_{n+1,m}$.

\begin{prop}\label{clarificado} Both $\Phi'$ and $\Phi''$ factor through the identification map
\begin{equation*}
\begin{array}{rcl}
\widetilde{\Sigma}(J^n(A) \times J^{m}(B)) &\stackrel{\mathrm{pr}}
\longrightarrow& (A^{n+1}\times B^{m+1})\wedge \disfrac{\Delta^n\times \Delta^m \times I}{\partial(\Delta^n\times \Delta^m \times I)} \\ \left[ \left(\langle \mathbf{a}\mid \mathbf{t}\rangle, \langle \mathbf{b}\mid \mathbf{s}\rangle\right), u \rule{0mm}{3mm}\right] & \longmapsto & (\mathbf{a},\mathbf{b})\wedge [\mathbf{t},\mathbf{s},u]
\end{array}
\end{equation*}
and the resulting maps $$\widehat{\Phi}',\widehat{\Phi}''\colon (A^{n+1}\times B^{m+1})\wedge \disfrac{\Delta^n\times \Delta^m \times I}{\partial(\Delta^n\times \Delta^m \times I)} \to \left((A\times B)^{n+m+2}\right)\wedge \disfrac{\Delta^{n+m+1}}{\partial(\Delta^{n+m+1})}$$ are described up to homotopy by
\begin{eqnarray}
\widehat{\Phi}'&\simeq&\sum\limits_{\theta\in {\mathcal S}_{n,m+1}}(-1)^{|\theta|}\widetilde{\Sigma}^{n+m+1}\delta_{\theta}\circ(\mathrm{Id}_{A^{n+1}}\times \mathrm{Id}_{B^m}\times i_1)\label{lafiprima}\\
\widehat{\Phi}''&\simeq&(-1)^{m}
\sum\limits_{\theta\in {\mathcal S}_{n+1,m}}(-1)^{|\theta|}\widetilde{\Sigma}^{n+m+1}\delta_{\theta}\circ(\mathrm{Id}_{A^{n}}\times i_1\times \mathrm{Id}_{B^{m+1}})\label{lafidobleprima}
\end{eqnarray}
where $i_1:Z\hookrightarrow Z\times Z$ is the inclusion $i_1(z)=(z,*)$.
\end{prop}

\begin{proof}
The asserted factorizations of $\Phi'$ and $\Phi''$ follow from an easy check. Furthermore, {$\widehat{\Phi}'$ factors as}
$$
\resizebox{1\textwidth}{!}{
\xymatrix{
(A^{n+1}\times B^{m+1})\wedge \disfrac{\Delta^n\times \Delta^m \times I}{\partial(\Delta^n\times \Delta^m \times I)}\ar[d]^{\mathrm{Id}\wedge \bar h_{n,m+1}}\\
(A^{n+1}\times B^{m+1})\wedge \disfrac{\Delta^n\times \Delta^{m+1}}{\partial(\Delta^n\times \Delta^{m+1})}\ar[r]^{\mathrm{Id}\wedge\nu_{n,m+1}}&
\bigvee_{\theta\in{\mathcal S}_{n,m+1} }(A^{n+1}\times B^{m+1})\wedge \disfrac{\Delta_{\theta}}{\partial(\Delta_{\theta})}\ar[d]^{(\mathrm{Id}_{A^{n+1}}\times \mathrm{Id}_{B^{m}}\times i_1)\wedge{\mathrm{Id}}}\\
&\bigvee_{\theta\in{\mathcal S}_{n,m+1} }(A^{n+1}\times B^{m+2})\wedge \disfrac{\Delta_{\theta}}{\partial(\Delta_{\theta})}\ar[d]^{(\delta_{\theta}\wedge \bar\psi_{\theta})}\\
&\widetilde{\Sigma}^{n+m+1}(A\times B)^{n+m+2}=((A\times B)^{n+m+2})\wedge \disfrac{\Delta^{n+m+1}}{\partial(\Delta^{n+m+1})}
}}
$$
{where the maps $\bar\psi_{\theta}$ are as in~(\ref{diagramPinchmap}), and} $\bar h_{n,m+1}$ is induced by the (surjective) map $h_{n,m+1}:\Delta^n\times \Delta^m\times I \to  \Delta^n\times \Delta^{m+1}$ given by $h_{n,m+1}(\mathbf{t},\mathbf{s},u)=(\mathbf{t},(1-u)\mathbf{s},u)$. Since the maps $\bar h_{n,m+1}$ and $\bar\psi_{\theta}$ are of degree~$1$, {(\ref{inducidoabajo}) implies that the above composition is homotopic to the right-hand term in~(\ref{lafiprima}). The equivalence in~(\ref{lafidobleprima}) is obtained in the same manner, now using the (surjective) map $h_{n+1,m}:\Delta^n\times \Delta^m\times I \to \Delta^{n+1}\times \Delta^{m}$ given by $h_{n+1,m}((\mathbf{t},\mathbf{s},u))=((1-u)\mathbf{t},u,\mathbf{s})$}
%
and its induced map $\bar h_{n+1,m}$ which is of degree $(-1)^m$.
\end{proof}

Proposition~\ref{clarificado}, the triangle in~(\ref{difzetadif}), and the functoriality of standard difference pinch maps yields:
\begin{cor} There is a commutative diagram
$$\xymatrix{
J^n(A)\circledast J^m(B) \ar[d]_-{\mathrm{pr}\,\cdot\,\zeta} \ar[rr]^-{\Phi^{A,B}_{n,m}}& &J^{n+m+1}(A\times B)\ar[d]^-r\\
(A^{n+1}\times B^{m+1})\wedge \disfrac{\Delta^n\times \Delta^m \times I}{\partial(\Delta^n\times \Delta^m \times I)} \ar[rr]& &\widetilde{\Sigma}^{n+m+1}(A\times B)^{n+m+2}
}$$
where the bottom horizontal map is obtained by subtracting~(\ref{lafiprima}) from~(\ref{lafidobleprima}).
\end{cor}

\subsection{Homology computations}
\label{homologyComp}

{In this section we let $R$ stand for a} commutative ring, and assume that  $H_*(A;R)$ and $H_*(B;R)$ are free graded $R$-modules of finite type. {We use the short hand $H(Z):=H_*(Z;R)$ and $\widetilde{H}(Z):=\widetilde{H}_*(Z;R)$,} and apply the standard K\" unneth formulae to identify the homology {(respectively reduced homology)} of {Cartesian} products {(respectively smash products)} with the tensor product of {the homologies (respectively reduced homologies) of the factors. In particular, the map induced in reduced homology by the identification map $X\times Y\to X\wedge Y$ corresponds under the isomorphism $\widetilde{H}(X\times Y)\cong\widetilde{H}(X)\otimes R\oplus R\otimes\widetilde{H}(Y)\oplus\widetilde{H}(X)\otimes\widetilde{H}(Y)$ with projection onto the third summand.}

\begin{thm}\label{morhipsmafterdesusp} Let $\rho:(A\times B)^{n+m+2} \to (A\times B)^{\wedge n+m+2}$ be the identification map. {Up to an automorphism of $\widetilde{H}(A \times B)^{\otimes n+m+2}$,} the morphism induced in {reduced} homology
$$
s^{-(n+m+1)}(\Phi_{n,m}^{A,B})_*:\widetilde{H}(A)^{\otimes n+1} \otimes \widetilde{H}(B)^{\otimes m+1}\to \widetilde{H}(A\times B)^{\otimes n+m+2}
$$
is given by {the restriction of}
\begin{multline*}
(-1)^{m}
\sum\limits_{\theta\in {\mathcal S}_{n+1,m}}(-1)^{|\theta|}(\rho\delta_{\theta})_*\circ(\mathrm{Id}_{A^{n}}\times i_1\times \mathrm{Id}_{B^{m+1}})_*\\
-\sum\limits_{\theta\in {\mathcal S}_{n,m+1}}(-1)^{|\theta|}(\rho\delta_{\theta})_*\circ(\mathrm{Id}_{A^{n+1}}\times \mathrm{Id}_{B^m}\times i_1)_*.
\end{multline*}
\end{thm}

\begin{proof} Let $\tilde\rho$ denote $\widetilde{\Sigma}^{n+m+1}\rho$. Since the sequence of identification maps
$$
\xymatrix{
J^{n+m+1}(A\times B)\ar[r]^-r &\widetilde{\Sigma}^{n+m+1}(A\times B)^{n+m+2} \ar[r]^-{\tilde\rho} & \widetilde{\Sigma}^{n+m+1}(A\times B)^{\wedge n+m+2}
}
$$
is a homotopy equivalence, the morphism induced by $\Phi_{n,m}^{A,B}$ in {reduced} homology coincides, {up to an automorphism,} with the morphism induced by {the composition $\mathrm{pr}\circ \zeta$ followed by the morphism induced by} the difference
\begin{multline*}
(-1)^{m}
\sum\limits_{\theta\in {\mathcal S}_{n+1,m}}(-1)^{|\theta|}\widetilde{\Sigma}^{n+m+1}\rho\delta_{\theta}\circ(\mathrm{Id}_{A^{n}}\times i_1\times \mathrm{Id}_{B^{m+1}})\\
-\sum\limits_{\theta\in {\mathcal S}_{n,m+1}}(-1)^{|\theta|}\widetilde{\Sigma}^{n+m+1}\rho \delta_{\theta}\circ(\mathrm{Id}_{A^{n+1}}\times \mathrm{Id}_{B^m}\times i_1).
\end{multline*}
{The result follows since the composition $\mathrm{pr}\circ\zeta$ induces the ($n+m+1$)-suspension of the inclusion $\widetilde{H}(A)^{\otimes n+1}\otimes \widetilde{H}(B)^{\otimes m+1}\hookrightarrow\widetilde{H}(A^{n+1}\times B^{m+1})$, and because} all the suspensions involved are $(n+m+1)${-fold} suspensions, so we can globally desuspend without altering signs.
\end{proof}

We now compute {the morphism induced by $\Phi^{A,B}_{n,m}$} in some particular cases.
Recall that, if $z\in  H(Z)$ is a primitive class, then, for any $k\geq 1$,
$$(\Delta_k^Z)_*(z)=\sum\limits_{i=1}^k 1\otimes \cdots 1 \otimes z_i\otimes 1\otimes \cdots \otimes 1$$
where $z_i=z$. For $k=1$, $\Delta_1^Z=\mathrm{Id}_Z$.

\begin{exs}
\begin{enumerate}[(a)]
\item Let $n=m=0$. In this case ${\mathcal S}_{1,0}=\{(1||\hspace{1.3mm})\}$ and $\hspace{.5mm}{\mathcal S}_{0,1}=\{(\hspace{1.3mm}||1)\}$ are both reduced to a single element. Thus, after a single desuspension, the morphism induced by $\Phi_{0,0}^{A,B}$ is given by\footnote{From this point on we omit any reference {to} the fact that we really want the obvious restriction of the morphism, and that each of the given descriptions is up to an automorphism.}
$$
\rho_*(T_2)_*(i_1\times\Delta_2^B)_*-\rho_*(T_2)_*(\Delta_2^A\times i_1)_*.
$$
For $x\in \widetilde{H}(A) $ and $y \in \widetilde{H}(B)$ primitive homology classes, we have
\begin{align*}
(T_2)_*(i_1\times\Delta^B_2)_*(x\otimes y)& {}=(T_2)_*((x\otimes 1)\otimes(1\otimes y+y\otimes 1))\\
&{}=(x\otimes 1)\otimes (1\otimes y)+(x\otimes y)\otimes (1\otimes 1).
\end{align*}
Therefore applying $\rho_*$, only the first term $(x\otimes 1)\otimes (1\otimes y)$ remains.

In the same manner,
\begin{align*}
(T_2)_*(\Delta_2^A\times i_1)_*(x\otimes y) & {}=(T_2)_*((1\otimes x+x\otimes 1)\otimes(y\otimes 1))\\
 & {}=(-1)^{|x||y|}(1\otimes y)\otimes (x\otimes 1)+(x\otimes y)\otimes (1\otimes 1).
 \end{align*}
Applying $\rho_*$, only the term $(-1)^{|x||y|}(1\otimes y)\otimes (x\otimes 1)$ remains and we finally get
$$s^{-1}(\Phi_{0,0}^{A,B})_*(x\otimes y)=(x\otimes 1)\otimes (1\otimes y)-(-1)^{|x||y|}(1\otimes y)\otimes (x\otimes 1).$$
\item Let $n=1$ and $m=0$. In this case, we get the following expression for $x_1,x_2\in \widetilde{H}(A) $ and $y \in \widetilde{H}(B)$ primitive homology classes:
\begin{align*}
s^{-2}(\Phi_{1,0}^{A,B})_*(x_1\otimes x_2\otimes y)
& = (x_1\otimes 1)\otimes (x_2\otimes 1)\otimes (1\otimes y)\\
&\quad -(-1)^{|x_2||y|}(x_1\otimes 1)\otimes (1\otimes y)\otimes (x_2\otimes 1)\\
&\quad +(-1)^{(|x_1|+|x_2|)|y|}(1\otimes y)\otimes (x_1\otimes 1)\otimes (x_2\otimes 1).
\end{align*}
\end{enumerate}
\end{exs}

In order to generalize these computations (in Proposition~\ref{shuffleinhomology} below), we need to fix some preliminary notation. For a permutation $\sigma\in\mathcal{S}_k$ and non-negative integers $d_1,\ldots,d_k$, let $(-1)^{s(\sigma,d_1,\ldots,d_k)}$ be the sign introduced by the morphism induced in reduced homology by maps $$\widetilde{\sigma}: Z^{\wedge k}\to Z^{\wedge k}, \hspace{2mm}(z_1,\ldots,z_k)\mapsto(z_{\sigma^{-1}(1)},\ldots,z_{\sigma^{-1}(k)}),$$ upon evaluation at a tensor $v_1\otimes \cdots \otimes v_k\in \widetilde{H}(Z)^{\otimes k}$ with $|v_i|=d_i$. Explicitly,
$$
\widetilde{\sigma}_*(v_1\otimes \cdots \otimes v_k)=(-1)^{s(\sigma, d_1,\ldots, d_k)}\hspace{.5mm} v_{\sigma^{-1}(1)}\otimes \cdots \otimes v_{\sigma^{-1}(k)}
$$
where $s(\sigma, d_1,\ldots, d_k)=\sum_{(i,j)\in I_\sigma}d_id_j$ and $I_\sigma=\{(i,j)\colon i<j \mbox{ and }\sigma(j)<\sigma(i)\}$ is the set of $\sigma$-inversions. Since we only care about the mod 2 value of $s(\sigma, d_1,\ldots, d_k)$, we can think of $s(\sigma, d_1,\ldots, d_k)$ as the number of shiftings that contribute with a $-1$ sign in permuting $v_1\otimes \cdots \otimes v_k$ to $v_{\sigma^{-1}(1)}\otimes \cdots \otimes v_{\sigma^{-1}(k)}$; that is, the number of $\sigma$-inversions $(i,j)$ for which $d_id_j$ odd.

\begin{prop}\label{shuffleinhomology}
Let $x_i\in \widetilde{H}(A) $ and $y_j \in \widetilde{H}(B)$ be primitive homology classes. The morphism
$${s^{-(n+m+1)}(\Phi_{n,m}^{A,B})_*\colon{}} \widetilde{H}(A)^{\otimes n+1} \otimes \widetilde{H}(B)^{\otimes m+1}\to
\widetilde{H}(A\times B)^{\otimes n+m+2}$$
takes $x_1\otimes \cdots \otimes x_{n+1}\otimes y_{n+2}\otimes \cdots \otimes y_{n+m+2}$ to 
$$(-1)^{m}\sum\limits_{\sigma \in {\mathcal S}_{n+1,m+1}}(-1)^{\varepsilon(\sigma)}z_{\sigma^{-1}(1)}\otimes \cdots \otimes z_{\sigma^{-1}(n+m+2)}$$
where ${\mathcal S}_{n+1,m+1}$ denotes the set of $(n+1,m+1)$ shuffles,
$$
z_{i}=\begin{cases}
x_{i}\otimes 1, & \mbox{ if } 1\leq i \leq n+1; \\
1\otimes y_{i}, & \mbox{ if } n+2\leq i \leq n+m+2,
\end{cases}
$$
and ${\varepsilon}(\sigma)=|\sigma|+s(\sigma,|x_1|,\ldots,|x_{n+1}|,|y_{n+2}|,\ldots,|y_{n+m+2}|)$.
\end{prop}

In other words, the sum of the statement can be written as
$$(-1)^{m}\hspace{-2.5mm}\sum\limits_{\sigma \in {\mathcal S}_{n+1,m+1}}\hspace{-2.3mm} (-1)^{|\sigma|}\hspace{.6mm}{\widetilde{\sigma}_*} ((x_1\otimes1)\otimes \cdots \otimes (x_{n+1}\otimes1)\otimes (1\otimes y_{n+2})\otimes\cdots\otimes (1\otimes y_{n+m+2})).$$

\begin{proof} We first observe that a $(n+1,m+1)$ shuffle $\sigma$ satisfies either $\sigma(n+m+2)=n+m+2$ or $\sigma(n+1)=n+m+2$. Then the sum of the statement splits in two parts:
\begin{multline}\label{sumstatementshuffle}
(-1)^{m}\sum\limits_{\sigma(n+m+2)=n+m+2} (-1)^{{\varepsilon}(\sigma)}z_{\sigma^{-1}(1)}\otimes \cdots \otimes z_{\sigma^{-1}(n+m+2)}\\
+ (-1)^{m}\sum\limits_{\sigma(n+1)=n+m+2} (-1)^{{\varepsilon}(\sigma)}z_{\sigma^{-1}(1)}\otimes \cdots \otimes z_{\sigma^{-1}(n+m+2)}.
\end{multline}

These two parts can be identified with the two parts of the morphism described in Theorem~\ref{morhipsmafterdesusp}. Indeed, if $\sigma(n+m+2)=n+m+2$ then $\sigma$ is completely determined by the  $(n+1,m)$ shuffle $\theta$ given by $\theta(i)=\sigma(i)$ and we have $|\sigma|=|\theta|$ and
\begin{multline*}
s(\sigma,|x_1|,\ldots,|x_{n+1}|,|y_{n+2}|,\ldots,|y_{n+m+2}|)\\
=s(\theta,|x_1|,\ldots,|x_{n+1}|,|y_{n+2}|,\ldots,|y_{n+m+1}|).
\end{multline*}
Then we have
\begin{multline*}
(-1)^{m}\sum\limits_{\sigma(n+m+2)=n+m+2} (-1)^{{\varepsilon}(\sigma)}z_{\sigma^{-1}(1)}\otimes \cdots \otimes z_{\sigma^{-1}(n+m+2)}\\
= (-1)^{m}\sum\limits_{\theta \in {\mathcal S}_{n+1,m}} (-1)^{|\theta|}\,\widetilde{\theta}_*(z_1\otimes \cdots \otimes z_{n+m+1})\otimes (1\otimes y_{n+m+2}).
\end{multline*}

On the other hand, the first part of the morphism described in Theorem~\ref{morhipsmafterdesusp} applied to $x_1\otimes \cdots \otimes x_{n+1}\otimes y_{n+2}\otimes \cdots \otimes y_{n+m+2}$ gives
$$(-1)^{m}
\sum\limits_{\theta\in {\mathcal S}_{n+1,m}}(-1)^{|\theta|}(\rho\delta_{\theta})_*(x_1\otimes \cdots \otimes x_{n+1}\otimes 1\otimes y_{n+2}\otimes \cdots \otimes y_{n+m+2})$$
and a straightforward calculation gives
\begin{multline*}
(\rho\delta_{\theta})_*(x_1\otimes \cdots \otimes x_{n+1}\otimes 1\otimes y_{n+2}\otimes \cdots \otimes y_{n+m+2})\\
= \widetilde{\theta}_*(z_1\otimes \cdots \otimes z_{n+m+1})\otimes (1\otimes y_{n+m+2})
\end{multline*}
which identifies the first part of $(\ref{sumstatementshuffle})$ with the first part of the morphism of
Theorem~\ref{morhipsmafterdesusp}.

We now suppose that $\sigma(n+1)=n+m+2$. Then $\sigma$ is completely determined by the  $(n,m+1)$ shuffle $\theta$ given by $\theta(i)=\sigma(i)$ for $1\leq i\leq n$ and $\theta(i)=\sigma(i+1)$ for $n+1\leq i\leq n+m+1$. We have $|\sigma|=|\theta|+m+1$ and
\begin{multline*}
s(\sigma,|x_1|,\ldots,|x_{n+1}|,|y_{n+2}|,\ldots,|y_{n+m+2}|)\\
=s(\theta,|x_1|,\ldots,|x_{n}|,|y_{n+2}|,\ldots,|y_{n+m+2}|)+|x_{n+1}|\left(|y_{n+2}|+\cdots+|y_{n+m+2}|\right).
\end{multline*}
Writing $e=|x_{n+1}|(|y_{n+2}|+\cdots+|y_{n+m+2}|)$, the second part of $(\ref{sumstatementshuffle})$ becomes
\begin{align*}
& (-1)^{m}  \sum\limits_{\sigma(n+1)=n+m+2} (-1)^{{\varepsilon}(\sigma)}z_{\sigma^{-1}(1)}\otimes \cdots \otimes z_{\sigma^{-1}(n+m+2)}\\
 & = -\sum\limits_{\theta \in {\mathcal S}_{n,m+1}}(-1)^{\varepsilon(\theta)+ e}
  (z_{\sigma^{-1}(1)}\otimes \cdots \otimes z_{\sigma^{-1}(n+m+1)})\otimes(x_{n+1}\otimes 1)\\
 & = -\sum\limits_{\theta \in {\mathcal S}_{n,m+1}} (-1)^{|\theta|+e}\;{\widetilde{\theta}_*}
(z_1\otimes \cdots \otimes z_{n+m+1})\otimes(x_{n+1}\otimes 1)\\
& = -\sum\limits_{\theta \in {\mathcal S}_{n,m+1}} (-1)^{|\theta|}(\rho\delta_{\theta})_*(x_1\otimes \cdots \otimes x_{n+1}\otimes y_{n+2}\otimes \cdots \otimes y_{n+m+2}\otimes 1),
\end{align*}
which completes the proof.
\end{proof}

\red{We are now ready for:}

\begin{proof}[Proof of Theorem~\ref{eli}]
In terms of the functorial homotopy equivalence $$F_{n+m+1}(Z)=J^{n+m+1}(\Omega Z)\simeq\widetilde{\Sigma}^{n+m+1}(\Omega Z)^{\wedge n+m+2},$$ the map
$\bar{\chi}_{n+m+1}: F_{n+m+1}(S^p\times S^p) \to F_{n+m+1}(S^p)$ takes the form
$$\widetilde{\Sigma}^{n+m+1}  \bar{\chi}^{\wedge n+m+2}\colon (\Omega S^p\times \Omega S^p)^{\wedge n+m+2}\wedge\disfrac{\Delta^{n+m+1}}{\partial \Delta^{n+m+1}} \to (\Omega S^p)^{\wedge n+m+2}\wedge \disfrac{\Delta ^{n+m+1}}{\partial \Delta^{n+m+1}}.$$
Since this is a $(n+m+1)$-fold suspension, we have
$$s^{-({n+m+1})}(\bar{\chi}_{n+m+1}\circ \Phi_{n,m}^{\Omega S^p,\Omega S^p})_*=(\bar{\chi}_*)^{\otimes n+m+2} \circ s^{-({n+m+1})}(\Phi_{n,m}^{\Omega S^p,\Omega S^p})_{*}$$
where $\bar{\chi}:\Omega S^p\times \Omega S^p\to \Omega S^p$ is given by $(\alpha,\beta)\mapsto \alpha^{-1}\beta$. We take $R=\Z$. {Recall that} $H_*(\Omega S^p)=\Z[x]$ where $x$ is primitive with degree $|x|=p-1$, {and that} the morphism
$$\bar{\chi}_*:H_*(\Omega S^p \times\Omega S^p) \to H_*(\Omega S^p)$$
is given by $\bar{\chi}_*(x\times 1)=-x$ and   $\bar{\chi}_*(1\times x)=x$.

By Proposition \ref{shuffleinhomology}, we know that $s^{-({n+m+1})}(\Phi_{n,m}^{\Omega S^p,\Omega S^p})_{*}$ takes the generator $x^{\otimes n+m+2}$ of $H_*(\Omega S^p)^{\otimes n+m+2}$ to
$$(-1)^{m}\sum\limits_{\sigma \in {\mathcal S}_{n+1,m+1}} (-1)^{{\varepsilon}(\sigma)}z_{\sigma^{-1}(1)}\otimes \cdots \otimes z_{\sigma^{-1}(n+m+2)}$$
where $z_i=x\otimes 1$ for $1\leq i\leq n+1$ and $z_i=1\otimes x$ for $n+2\leq i\leq n+m+2$.

Since in each term there are exactly $n+1$ components $x\times 1$ we obtain
$$(\bar{\chi}_*)^{\otimes n+m+2} \circ s^{-({n+m+1})}(\Phi_{n,m}^{\Omega S^p,\Omega S^p})_{*}(x^{\otimes n+m+2})=\pm\hspace{-3mm}\sum\limits_{\sigma \in {\mathcal S}_{n+1,m+1}} (-1)^{{\varepsilon}(\sigma)}x^{\otimes n+m+2}.$$

Now, if $\sigma$ is a permutation of positive signature, we have, for some integer $k$ {(depending on $\sigma$),}
$${\varepsilon}(\sigma)=|\sigma|+s(\sigma,\underbrace{|x|,\ldots,|x|}_{n+m+2})=|\sigma|+2k|x||x|\equiv 0\quad mod \,\,2$$
while,  if $\sigma$ is a permutation of negative signature, we have, for some integer $k$ {(depending on $\sigma$),}
$${\varepsilon}(\sigma)=|\sigma|+s(\sigma,\underbrace{|x|,\ldots,|x|}_{n+m+2})=|\sigma|+(2k+1)|x||x|\equiv 1+|x|=p\quad mod \,\, 2.$$
This gives the result.
\end{proof}


\begin{thebibliography}{99}
\bibitem{Bau} {H. J. Baues.} Iterierte Join-Konstruktionen. \textit{Math. Z.} \textbf{131} (1973), 77–-84.

\bibitem{B-H} {I. Berstein \and P.J. Hilton}. Category and
generalized Hopf invariants. \textit{Illinois J. Math.} \textbf{4}
(1960), 437--451.

\bibitem{carrasquel}
{J. G. Carrasquel-Vera.}
The Ganea conjecture for rational approximations of sectional category, \textit{J.\ Pure Appl.\ Algebra},  \textbf{220} (2016), no. 4, 1310–-1315.

\bibitem{CLOT}
{O. Cornea, G. Lupton, J. Oprea, \and D. Tanr\'e.} \textit{Lusternik-Schnirelmann category}, AMS Mathematical Surveys and Monographs 103 (2003).

\bibitem{Dran} {A. Dranishnikov.} Topological complexity of wedges and covering maps, \textit{Proc.\ Amer.\ Math.\ Soc.} \textbf{142} (2014), no. 12, 4365–-4376.

\bibitem{EilenbergSteenrod} S. Eilenberg, N. Steenrod, \textit{Foundations of Algebraic Topology}, Princeton, Princeton University Press, 1952.

\bibitem{Far} {M. Farber.} Topological complexity of motion
planning. \textit{Discrete Comput. Geom.} \textbf{29} (2003),
211--221.

\bibitem{FarInv} {M. Farber.}
\newblock {\em Invitation to topological robotics}.
\newblock Zurich Lectures in Advanced Mathematics. European Mathematical
  Society (EMS), Z\"urich, 2008.

\bibitem{FarCosta} {M. Farber \and A. Costa.} Motion planning in spaces with small fundamental groups, \textit{Commun. Contemp. Math.}, \textbf{12}, no.1 (2010), 107--119.

\bibitem{FG} {M. Farber \and M. Grant.} Robot motion planning, weights of cohomology classes, and
              cohomology operations, \textit{Proc. Amer. Math. Soc.}, \textbf{136} (2008), 3339--3349.

\bibitem{FG2} {M. Farber \and M. Grant.} Topological complexity of configuration spaces, \textit{Proc. Amer. Math. Soc.} \textbf{137} (2009), no. 5, 1841–-1847.

\bibitem{FTY} {M. Farber, S. Tabachnikov \and S. Yuzvinsky.} Topological robotics: motion planning in projective spaces.
   \textit{Int. Math. Res. Not.}, \textbf{34} (2003), 1853--1870.

\bibitem{FGKV} {L. Fern\'andez Su\'arez, P. Ghienne, T. Kahl and L. Vandembroucq.} Joins of DGA modules and sectional category. \textit{Algebr. Geom. Topol.}, \textbf{6} (2006), 119--144.

\bibitem{Fox} {R. H. Fox.} On the Lusternik-Schnirelmann category, \textit{Ann. of Math.} (2) \textbf{42}, (1941), 333-–370.

\bibitem{Ganea1} {T. Ganea.}  Some problems on numerical homotopy invariants. \textit{In: Symposium on Algebraic Topology , 1971, Lecture Notes in Math.} \textbf{249} (1971), 23--30.

\bibitem{Weaksecat} {J. M. Garc\'{\i}a Calcines and L.
Vandembroucq}. Weak sectional category, \textit{Journal of the
London Math. Soc.} \textbf{82}(3) (2010), 621--642.

\bibitem{GC-V} {J. M. Garc\'{\i}a Calcines and L.
Vandembroucq}. Topological complexity and the homotopy cofibre of
the diagonal map, \textit{Math. Z.} \textbf{274} (1-2) (2013),
145--165.

\bibitem{GonzalezGrant} {J. Gonz\'alez \and M. Grant.} Sequential motion planning of non-colliding particles in Euclidean spaces, \textit{Proc. Amer. Math. Soc.}~\textbf{143} (2015) 4503-4512.

\bibitem{GonzalezGrantVandembroucq} {J. Gonz\'alez \and M. Grant \and L. Vandembroucq.} Hopf invariants, topological complexity, and LS-category of the cofiber of the diagonal map for two-cell complexes, to appear in \textit{Contemporary Mathematics}.

\bibitem{Grant} {M. Grant}. Topological complexity, fibrations and symmetry, \textit{Topology Appl.} \textbf{159} (2012), 88--97.

\bibitem{GLO} {M. Grant, G. Lupton \and J. Oprea.} Spaces of topological complexity one, \textit{Homology Homotopy Appl.} \textbf{15} (2013), No. 2, 73--81.

\bibitem{GLO2} {M. Grant, G. Lupton \and J. Oprea.} A mapping theorem for topological complexity, \textit{Algebr.\ Geom.\ Topol.}, {\bf 15} (2015), 1643--1666.

\bibitem{Harper} J. Harper. Category and Products,  \textit{Rend. Sem. Mat. Fis. Milano} \textbf{68} (1998), 165-–177.

\bibitem{Hatcher} A. Hatcher. \textit{Algebraic Topology}, Cambridge University Press, 2001.

\bibitem{Hess} K. Hess. A proof of Ganea's conjecture for rational spaces, \textit{Topology} \textbf{30} (1991), no. 2, 205–-214.

\bibitem{Iwase} N.  Iwase.  Ganea's conjecture on Lusternik-Schnirelmann
category, \textit{Bull.  London Math.  Soc.} \textbf{30} (6) (1998) 623--634.

\bibitem{IwaseAinfty} N.  Iwase.  $A_{\infty}$ methods in Lusternik-Schnirelmann category, \textit{Topology} \textbf{41} (2002), 695--723.

\bibitem{James} I. M. James. On category, in the sense of Lusternik-Schnirelmann. \textit{Topology} \textbf{17} (1978), no. 4, 331--348.

\bibitem{Jessup} B. Jessup. Rational approximations to L-S category and a conjecture of Ganea, \textit{Trans. Amer. Math. Soc.} \textbf{317} (1990), no. 2, 655-–660.

\bibitem{JMP} {B. Jessup, A. Murillo \and P-E. Parent.}
Rational topological complexity.
\textit{Algebr. Geom. Topol.} \textbf{12} (2012), no. 3, 1789-–1801.

\bibitem{Marcum} {H. J. Marcum.} Fibrations over double mapping cylinders, \textit{Illinois J.\ Math.} \textbf{24} (1980), no. 2, 344--358.

\bibitem{Rudyak} {Y. B. Rudyak.} On the Ganea conjecture for manifolds, \textit{Proc. Amer. Math. Soc.} \textbf{125} (1997), no. 8, 2511--2512.

\bibitem{Schwarz} {A. Schwarz.} The genus of a fiber space,
\textit{A.M.S. Transl.} \textbf{55} (1966), 49--140.

\bibitem{Singhof} W. Singhof, Minimal coverings of manifolds with balls, \textit{Manuscripta Math.} \textbf{29} (1979), no. 2-4, 385–-415.

\bibitem{Spanier} {E. Spanier.} {\em Algebraic Topology}, McGraw-Hill, 1966.

\bibitem{Stanley} {D. Stanley.} On the Lusternik--Schnirelmann category of maps. \textit{Canad. J. Math.} \textbf{54} (2002), 608--633.

\bibitem{strom-ls-on-manifolds} {J. Strom} Two special cases of Ganea's conjecture, \textit{Trans.~Amer.~Math.~Soc.} \textbf{352} (2000), no.~2, 679--688.

\bibitem{Strom} {J. Strom.} Decomposition of the diagonal map, \textit{Topology} \textbf{42} (2003), no. 2, 349--364.

\bibitem{MR2839990} {J. Strom.}
\textit{Modern classical homotopy theory},
Graduate Studies in Mathematics 127,
American Mathematical Society, Providence, RI,  2011.

\bibitem{whiteheadbookelements}{G.W.~Whitehead.}
\textit{Elements of homotopy theory},
Graduate Texts in Mathematics 61,
Springer-Verlag, New York-Berlin, 1978.

\end{thebibliography}
\end{document}